\documentclass[11pt,reqno]{amsart}
\usepackage{bcol}

\newcommand{\R}{\ensuremath{\mathbb{R}}}

\newcommand{\rev}[1]{{#1}}

\usepackage{amssymb}
\usepackage{geometry}
\usepackage{mathtools}
\usepackage{hyperref}
\usepackage{multirow}
\usepackage{array}
\usepackage[ruled,vlined]{algorithm2e}

\usepackage[colorinlistoftodos,prependcaption,textsize=small]{todonotes}

\newtheorem{prop}{Proposition}

\theoremstyle{definition}

\theoremstyle{remark}
\newtheorem{remark}{Remark}
\newtheorem{example}{Example}

\def\1{\mathbf{1}}
\def\0{\mathbf{0}}

\newcommand{\conv}{\ensuremath{\operatorname{conv}}}

\newcommand{\diag}{\operatorname{diag}}
\newcommand{\prob}{\operatorname{\mathbf{Prob}}}

\newcommand{\cP}{\ensuremath{\text{OPT}}}
\newcommand{\ropt}{\ensuremath{\text{COPT}}}

\newcommand{\exclude}[1]{}
\newcommand{\BD}[1]{\normalsize{\mathbf{#1}}}

\makeatletter
\newcommand{\leqnomode}{\tagsleft@true}
\newcommand{\reqnomode}{\tagsleft@false}
\makeatother

\author{Alper Atamt\"{u}rk and Hyemin Jeon}
\title{Lifted Polymatroid Inequalities for Mean-Risk Optimization with Indicator Variables}
\thanks{ \noindent \hskip -5mm
	A. Atamt\"urk: Industrial Engineering \& Operations Research, University of California, Berkeley, CA 94720-1777.
	\texttt{atamturk@berkeley.edu}   \\
	H. Jeon: Industrial Engineering \& Operations Research, University of California, Berkeley, CA 94720-1777. \texttt{hyemin.jeon@berkeley.edu }
}

\begin{document}
\maketitle

%\begin{abstract}
%\end{abstract}

%\listoftodos

\BCOLReport{17.01}%{Mathematical Programming}

\begin{abstract} \vskip -1cm
We investigate a mixed $0-1$ conic quadratic optimization problem with indicator variables arising in mean-risk optimization. 
The indicator variables are often used to model non-convexities such as fixed charges or cardinality constraints.
Observing that the problem reduces to a submodular function minimization for its binary restriction,
we derive three classes of strong convex valid inequalities by lifting the polymatroid inequalities on the binary variables.
Computational experiments demonstrate the effectiveness of the inequalities in strengthening the convex relaxations and,
thereby, improving the solution times for mean-risk problems with fixed charges and cardinality constraints significantly. \\

\noindent
\textbf{Keywords:} Risk, submodularity, polymatroid, conic integer optimization, valid inequalities.
\end{abstract}

% SIAM 90C57, 90C11, 90C46

\begin{center}
March 2017; 
October 2017; \rev{May 2018}
\end{center}

\section{Introduction}
\label{sec:intro}

Optimization problems with a conic quadratic objective arise often 
when modeling uncertainty with a mean-risk utility.
We motivate such a model for an investment problem with a parametric Value-at-Risk (VaR) minimization objective.
Given random variables $\ell_i, \  i \in N,$ representing the uncertain loss in asset $i$, 
let $y_i$ denote the amount invested in asset $i \in N$.
Then, for small  $\epsilon > 0$,  
minimizing the Value-at-Risk with confidence level $1-\epsilon$ is stated as
\begin{center}
\begin{minipage}{0.8\textwidth}
\leqnomode
\begin{align*}
\label{eq:VaR} \tag{VaR}
\begin{split}
 \zeta(\epsilon) = \min \ & \bigg \{ z :  \prob \left( \ell'y > z \right) \leq \epsilon, \ \ y \in Y \bigg \} , 
\end{split}
\end{align*}
\reqnomode
\end{minipage}
\end{center}
\noindent where losses greater than $\zeta(\epsilon)$ occur with probability no more than $\epsilon$. 
Here, $Y$ represents the set of feasible investments.
If $\ell_i$'s are independent normally distributed random variables with mean $\mu_i$ and variance $\sigma^2_i$, problem \eqref{eq:VaR} is equivalent to the following mean-risk optimization problem: %Second Order Cone Program (SOCP):
\begin{align} \label{eq:conic-opt}
\min \ & \bigg \{  \mu' y + \Omega \sqrt{\sum_{i \in N} \sigma_i^2 y_i^2} : \ y \in Y \bigg \}, 
\end{align}
where $\Omega = \Phi^{-1}(1-\epsilon)$ and $\Phi$ is the cumulative distribution function (c.d.f.) of the standard normal distribution \cite{Birge:SPbook}.
If only the mean and variance of the distribution are known, one can write a robust version by letting $\Omega = \sqrt{(1-\epsilon)/\epsilon}$, which provides an upper bound on the worst-case VaR
\cite{bertsimas.popescu:05,ghaoui.etal:03}.
Alternatively, if $\ell_i$'s are independent and symmetric with support $[u_i - \sigma_i, u_i + \sigma_i]$, then letting $\Omega = \sqrt{\ln(1/\epsilon)}$ gives an upper bound on
on the worst-case VaR as well \cite{BN:robust-mp}. Hence, under different assumptions on the random variable $\ell$, one arrives at different instances of the mean-risk model \eqref{eq:conic-opt} with a conic quadratic objective.
Ahmed \cite{ahmed:06} studies the complexity and tractability of various stochastic objectives for mean-risk optimization.

The objective of the mean-risk optimization problem \eqref{eq:conic-opt} is a conic quadratic function in $y$, hence convex. 
If the feasible set $Y$ is a tractable convex set as well, then \eqref{eq:conic-opt} is an efficiently-solvable convex optimization problem \cite{NN:convex-book}. In practice, though, most problems are accompanied
with non-convex side constraints, such as a restriction on the maximum number of non-zero variables or fixed charges \cite{AAG:matchup,AAG:cruise,AG:arcset, benson.saglam:14,B:quad-mip,bonami.lejeune:09,DKRA:ramping} that are needed to obtain more realistic and implementable solutions. To model such non-convexities it is convenient to introduce auxiliary binary variables $x_i, \ i \in N,$ to indicate whether $y_i$ is non-zero or not. 
The so-called \textit{on-off constraints} $0 \le y_i \leq u_i x_i$, where $u_i$ is an upper bound on $y_i$, $i \in N$, model whether asset $i$ is in the solution or not. By appropriately scaling $y_i$, we assume, without loss of generality, that $u_i = 1$ for all $i \in N$. The non-convexity introduced by the on-off constraints is a major challenge in solving practical mean-risk optimization problems. In order to address this difficulty, 
in this paper, we derive strong convex relaxations for the conic quadratic mixed-integer set with indicator variables:
\begin{align}
F = \bigg \{ (x,y,z) \in \{0,1\}^N \times \R^N_+ \times \R_+: \sigma + \sum_{i \in N} a_i y_i^2  \le z^2, \ \0 \le y \le x  \bigg \},
\end{align}
where $\sigma \ge 0$ and $a_i > 0, \ i \in N$. 
\ignore{
We also consider the optimization problem over $F$:
\begin{align*}
\text{(OPT)} \ \ \ \ \ 
 \min \bigg  \{ c'x + d'y + z: (x,y,z) \in F \bigg \},
\end{align*}
which will be useful in proving the validity of the inequalities for $F$.
}

Problem \eqref{eq:conic-opt} is a special case of the mean-risk optimization problem
%\begin{minipage}{0.8\textwidth}
%\leqnomode
\begin{align}
\label{eq:cov} % \tag{$\cP_{0}$}
\begin{split} 
\ \min  & \bigg \{ \mu' y + \Omega  \sqrt{y' Q y}: y \in Y \bigg \} 
%\text{s.t. } & y \in Y 
%                & 0 \leq y_i \leq x_i, i \in N \\
%                & x \in \{0,1\}^n
\end{split}
\end{align}
%\reqnomode
%\end{minipage}
with a positive semidefinite covariance matrix $Q$.
By decomposing $Q = V + D $, where $V, D \succeq 0 $ and $D$ is a diagonal matrix,
problem \eqref{eq:cov} is equivalently written as
\begin{center}
	\begin{minipage}{0.8\textwidth}
		\leqnomode
		\begin{align*}
		\label{eq:cqip-r} % \tag{$\cP'_{0}$}
		\begin{split}
		\min \ &   \mu y+  \Omega z \\
		\text{s.t. } & y'D y + s^2 \le z^2 \\
		& y' V y \le s^2 \\
		& y \in Y, \ z \in \R_+.
		%                & 0 \leq y_i \leq x_i, i \in N \\
		%                & x \in \{0,1\}^n.
		\end{split}
		\end{align*}
		\reqnomode
	\end{minipage}
\end{center}
Indeed, for high dimensional problems such a decomposition is readily available, as a low-rank factor covariance matrix $V$ is estimated separately from the residual variance matrix $D$ to avoid ill-conditioning  \cite{grinold2000active}. Observe that the  first constraint above is a conic quadratic with a diagonal matrix.
Therefore, the valid inequalities derived here for the diagonal case can be applied more generally in the presence of correlations after constructing a suitable diagonal relaxation. We provide computational experiments on the application of the results for the general case with correlations as well.

\subsubsection*{Literature}

Utilizing diagonal matrices is standard for constructing convex relaxations in binary quadratic optimization \cite{anstreicher2012convex, poljak1995convex}. In particular, for $x \in \{0,1\}^n$, 
\[
x'Qx \le z \iff x'(Q-D)x + \diag(D)'x \le z
\]
with a diagonal matrix $D$ satisfying $Q-D \succeq 0$.
This transformation is based on the ideal (convex hull) representation of the
separable quadratic term $x'Dx$ as a linear term $\diag(D)'x$ for $x \in \{0,1\}^n$.

A similar approach is also available for convex quadratic optimization with indicator variables. For $x \in \{0,1\}^n$ and $y \in \R^n \text{ s.t. } \0 \le y \le x$, we have
\[
y'Qy \le z \iff y'(Q-D)y + \diag(D)'t \le z, \ y_i^2 \le x_i t_i
\]
with $t \in \R^n_+$ \cite{akturk.atamturk.gurel:09,gunluk.linderoth:10}. This transformation is based on the ideal representation of each
quadratic term $D_{ii}y_i^2$ subject to on-off constraints as a linear term $D_{ii}t_i$ along with a rotated cone constraint $y_i^2 \le x_i t_i$.
Decomposing $Q$ for diagonalization is also studied for an effective application of 
linear perspective cuts \cite{frangioni.gentile:07}. 
%A slight generalization of the result on the diagonal case was introduced for the reformulation 
%$ x'Qx = x'(Q - B)x + x'Bx$ with $2 \times 2$ block-diagonal matrix $B$ \cite{jeon.linderoth:16}.  

For the conic quadratic constraint $\sqrt{x'Dx} \le z$, however, the terms are \textit{not} separable even for the diagonal case, and simple transformations as in the quadratic cases above are not sufficient to arrive at an ideal convex reformulation. For the pure binary case, Atamt\"{u}rk and Narayanan \cite{atamturk.narayanan:08} exploit the submodularity of the underlying set function to describe its 
convex lower envelope via polymatroid  inequalities. Atamt\"urk and G\'omez \cite{atamturk.gomez:16} \rev{describe a variety of applications for this model and}
give strong valid inequalities for the mixed $0-1$ case without the on-off constraints. 
The ideal (convex hull) representation for the conic quadratic mixed $0-1$ set with indicator variables $F$ remains an open question. We show, however, that exploiting the submodularity of the underlying set function for the $0-1$ restrictions is critical in deriving strong convex relaxations for $F$.  
\autoref{tab:review} summarizes the results for the related sets described above. 
In addition, general conic mixed-integer cuts \cite{AN:conicmir:ipco}, lift-and-project cuts \cite{CI:cmip}, disjunctive cuts \cite{terlaky:disj,KilincKarzan2015} are also applicable to the conic mixed-integer set $F$ considered here.

\begin{table}
\centering
\caption{Convex hull representations for $x \in \{0,1\}^n, y \in \R^n_+, z \in \R_+$.}
\label{tab:review}
{\setlength{\extrarowheight}{5pt}
\begin{tabular}{l|c|c}
\hline \hline
 %\multicolumn{1}{c|}{}
 & Separable Quadratic & Conic Quadratic \\
\hline
Pure $0-1$ & $x'Dx \le z$: \cite{anstreicher2012convex,poljak1995convex}  & $\sqrt{x'Dx } \le z$:  \cite{atamturk.narayanan:08} \\
Mixed $0-1$ & $y'Dy \le z, \ \0 \le y \le x$: \cite{akturk.atamturk.gurel:09,gunluk.linderoth:10} &
$\sqrt{y'Dy} \le z, \ \0 \le y \le x:$  \ \ ? \\
\hline \hline
\end{tabular}}
\end{table}

\subsubsection*{Notation}
Throughout, we denote by $\0$ the vector of zeroes,  by $\1$ the vector of ones, 
and by $e_i$ the $i$th unit vector.  $N := \{1, 2, \ldots, n\}$ and 
$[k] := \{1, 2, \ldots, k\}$. For a vector $a \in \R^N$, let $a(S) = \sum_{i \in S} a_i$, $S \subseteq N$.
We use $(\cdot)^+$ to denote $\max\{\cdot,0\}$.

\subsubsection*{Outline}
The remainder of the paper is organized as follows. In Section~\ref{sec:prelim} we review the polymatroid inequalities for
the binary restriction of the mean-risk problem and give a polynomial algorithm for an optimization problem over $F$. 
In Section~\ref{sec:cuts} we introduce three classes of convex valid 
inequalities for $F$ that are obtained from binary restrictions of $F$ through lifting the polymatroid inequalities. 
In Section~\ref{sec:computation} we present computational experiments performed for testing the effectiveness of the proposed inequalities in solving mean-risk optimization problems with on-off constraints. We conclude with a few final remarks in \autoref{sec:conclusion}.

\section{Preliminaries}	
\label{sec:prelim}

\subsection{Polymatroid inequalities}
\label{sec:prelim:polymatroid}

%In this section, we recall the main results from Atamt\"urk and Narayanan \cite{atamturk.narayanan:08} that will be used in this paper. 
Given $\sigma \ge 0$ and $a_i > 0, \ i \in N$,  consider the set
\begin{align}
K_\sigma = \bigg \{ (x,z) \in \{0,1\}^N \times \R_+:  \sqrt{\sigma + \sum_{i \in N} a_i x_i } \le z  \bigg \}.
\end{align}
Observe that $K_0$ is the binary restriction of $F$ obtained by setting $y = x$.
For a given permutation $\left((1),(2),\ldots,(n)\right)$ of $N$, let
\begin{align}
\sigma_{(k)}&= a_{(k)} + \sigma_{(k-1)}, \text{ and } \sigma_{(0)} = \sigma,\notag\\
\pi_{(k)}&=\sqrt{\sigma_{(k)}}-\sqrt{\sigma_{(k-1)}},\label{eq:definitionPi}
\end{align}
and define the \emph{polymatroid inequality} as
\begin{equation}
\label{eq:extendedPolymatroidInequality}
\sum_{i=1}^n \pi_{(i)}x_{(i)}\leq z - \sqrt{\sigma}.
\end{equation}
Let $\Pi_\sigma$ be the set of such coefficient vectors $\pi$ for \textit{all} permutations of $N$.
The set function defining $K_\sigma$ is non-decreasing submodular; therefore, $\Pi_\sigma$ form the extreme points of a polymatroid \cite{edmonds:70} and
 the convex hull of $K_\sigma$ is given by
the set of all polymatroid inequalities \cite{L:submodular-convex}.

\begin{prop}[Convex hull of $K_\sigma$]
	\label{prop:convexHullK}
	$$\text{conv}(K_\sigma)=\left\{(x,z)\in [0,1]^N\times \R_+:\pi'x \leq z - \sqrt{\sigma}, \;\; \forall \pi\in \Pi_\sigma \right\}.$$
\end{prop}

As shown by Edmonds  \cite{edmonds:70}, the maximization of a linear function over a polymatroid can be solved by the greedy algorithm; therefore, a point 
$(\bar x, \bar z) \in [0,1]^N \times \R_+$ can be separated from $\text{conv}(K_\sigma)$ via the greedy algorithm by sorting $\bar x_i, \ i \in N$ in non-increasing order in $O(n \log n)$ time.

\begin{prop}[Separation]
	\label{prop:separation}
	A point $(\bar{x},\bar{z})\not \in \text{conv}(K_\sigma)$ such that $\bar{x}_{(1)}\geq \bar{x}_{(2)} \geq \ldots \geq \bar{x}_{(n)}$ is separated from $\conv(K_\sigma)$ by inequality \eqref{eq:extendedPolymatroidInequality}.
\end{prop}

Atamt\"urk and Narayanan \cite{atamturk.narayanan:08} consider the  mixed-integer version of $K_\sigma$:
\begin{equation*}
L_\sigma=\left\{(x,y,z)\in \{0,1\}^N\times [0,1]^M \times \R_+: {\sqrt{\sigma + \sum_{i\in N}a_ix_i+\sum_{i\in M}c_iy_i^2}}\leq z\right\},
\end{equation*}
where $c_i > 0, \ i \in M$  and give valid inequalities for $L_\sigma$ based on 
the polymatroid inequalities.
Without loss of generality, the upper bounds of the continuous variables in $L_\sigma$ are set to one by scaling.
%For $T\subseteq M$, define $c(T): = \sum_{i \in T} c_i$.
\begin{prop}[Valid inequalities for $L_\sigma$]
	\label{prop:validIneqBounded1}
	For $T\subseteq M$ inequalities
	\begin{equation}
	\label{eq:polymatroidBoundedDominated}
    \pi'x +	\sqrt{\sigma + \sum_{i\in T}c_iy_i^2} \leq z, \quad \pi\in \Pi_{\sigma + c(T)}
	\end{equation}
	are valid for $L_\sigma$.
\end{prop}

Inequalities \eqref{eq:polymatroidBoundedDominated} \rev{are} used \rev{to} derive nonlinear valid inequalities for $F$ in Section~\ref{sec:cuts}.

\subsection{Optimization}
\label{sec:prelim:solve}

In this section, we consider the optimization problem
\begin{align*}
(\cP) \ \ \ \ \
\min \bigg  \{ c'x + d'y + \sqrt{\sigma + \sum_{i \in N} a_i y_i^2} :  \0 \le y \le x, \ x \in \{0,1\}^N, \ y \in \R_+^N \bigg \},
\end{align*}
which will be useful in proving the validity of the inequalities for $F$.
We characterize the optimal solutions and give a polynomial algorithm for (\cP).
We assume that $\sigma \ge 0, \ a_i > 0, \ i \in N$ to ensure a real-valued objective. 
Without loss of generality, we assume that 
$c_i > 0, \ i \in N$, otherwise, we may set $x_i$ to one;
$d_i < 0, \ i \in N$, otherwise, we may set $y_i$ to zero;  and
$c_i + d_i < 0, \ i \in N$, otherwise, we may set both $x_i$ and $y_i$ to zero.

Without loss of generality, assume that the variables are indexed so that
\[\frac{c_1 + d_1}{a_1} \leq  \frac{c_2 + d_2}{a_2} \leq \cdots \leq \frac{c_n+d_n}{a_n} \cdot \]

The following proposition shows that the binary part of an optimal solution to (\cP) is a vector of consecutive ones,
followed by consecutive zeroes.
%A very similar argument of the proof was given in the work of Shen et al. \cite{shen.coullard.daskin:03} which we repeat here.

\begin{prop} \label{prop:consec1}
If $(x^*, y^*)$ is an optimal solution to (\cP), then $x^*_k = 1$ for some $k \in N$, implies $x^*_i = 1$ for all $i \in [k-1]$.
\end{prop}

\begin{proof}
Suppose for contradiction that $x^*_k = 1$, but $x^*_j = 0$ for some $j < k$.
Consider two feasible points $(x',y')$ and $(x'', y'')$ 
with respective objective values $z'$ and $z''$, constructed as: 
\ignore{
\begin{alignat*}{3}
 x'_i &= \begin{cases} 1 & \ \text{if  } i = j \\ x^*_i & \text{otherwise} \end{cases} \quad
 y'_i &= \begin{cases} 1 & \ \text{if  } i = j \\ y^*_i & \text{otherwise} \end{cases}  \\
 x''_i &= \begin{cases} 0 & \ \text{if  } i = k \\ x^*_i & \text{otherwise} \end{cases} \quad
 y''_i &= \begin{cases} 0 & \ \text{if  } i = k \\ y^*_i & \text{otherwise.} \end{cases}
\end{alignat*}
}
\begin{align*}
(x', y') &= (x^*, y^*) + (e_j, e_j), \\
(x'', y'') &= (x^*, y^*) - (e_k, y^*_k e_k ).  
\end{align*}
We will show that $z' < z^*$, contradicting the optimality of $(x^*,y^*)$. To this end,
let $\xi := \sigma + \sum_{i \in N} a_i {y^*_i}^2$, and
\begin{alignat*}{3}
\delta_1 &:= z^* - z'' = c_k + d_k y^*_k+ \sqrt{\xi} - \sqrt{\xi - a_k {y^*_k}^2}, \\
\delta_2 &:= z' - z^* = c_j + d_j +  \sqrt{\xi+a_j} - \sqrt{\xi}.
\end{alignat*}
As $(x'', y'')$ is a feasible solution, $\delta_1 \leq 0$. Also note that $y_k^* > 0$
as otherwise $x_k^*$ would be zero in an optimal solution since $c_k > 0$.
Now, we establish that 
\begin{align*}
\frac{\delta_1}{a_k {y^*_k}^2} - \frac{\delta_2}{a_j}
= \left( \frac{c_k + d_k y^*_k}{a_k {y^*_k}^2} - \frac{c_j + d_j}{a_j}\right)
   + \left( \frac{\sqrt{\xi} - \sqrt{\xi-a_k{y^*_k}^2}}{a_k{y^*_k}^2} - \frac{\sqrt{\xi+a_j} - \sqrt{\xi}}{a_j} \right) > 0,
\end{align*}
from the inequality
\[\frac{c_k + d_k y^*_k}{a_k {y^*_k}^2}  \geq \frac{c_k + d_k}{a_k y^*_k} \geq \frac{c_j + d_j}{a_j}, \]
which holds by the indexing assumption and that $0 < y^*_k \le 1$, and from the inequality
 \[\frac{\sqrt{\xi} - \sqrt{\xi-a_k{y^*_k}^2}}{a_k{y^*_k}^2} - \frac{\sqrt{\xi+a_j} - \sqrt{\xi}}{a_j} > 0,\]
which follows from the strict concavity of square root function. Therefore, we have
$\frac{\delta_2}{a_j} < \frac{\delta_1}{a_k {y^*_k}^2} \leq 0$, implying $\delta_2 < 0$, which contradicts the optimality of $(x^*,y^*)$.
\end{proof}

\begin{prop}
	\label{prop:opt-compexity}
	There is an $O(n^2)$ time algorithm to solve (\cP). 
\end{prop}
\begin{proof}
\autoref{prop:consec1} implies that 
there exist only $n+1$ possible candidates for optimal $x$, i.e., $\mathbf{0}$ and $ \sum_{i=1}^{k} e_i$ for $k \in N$.
After a single sort of the indices  in $O(n \ log  \ n)$ time,  
for each candidate $x$ the resulting convex optimization problem in $y$ can be solved in $O(n)$ time
with \autoref{alg:KKT} in the Appendix . Therefore, an optimal solution to (\cP) can be found in $O(n^2)$ time.
\end{proof}

\section{Lifted Polymatroid Inequalities}
\label{sec:cuts}

In this section, we derive three classes of valid inequalities for $F$ by lifting 
the polymatroid inequalities \eqref{eq:extendedPolymatroidInequality} described in Section~\ref{sec:prelim:polymatroid} from
specific restrictions of the feasible set $F$.  The first class of inequalities are linear, whereas the other two are nonlinear convex inequalities.

\subsection{Lifted Linear Polymatroid Inequalities }

\label{sec:lift1}

Consider the restriction of $F$ obtained by setting the continuous variables $y$ to their binary upper bounds $x$. It 
follows from Section~\ref{sec:prelim:polymatroid} that for any permutation ((1),(2), \ldots, (n)) of $N$, 
the polymatroid inequality
\begin{align}
\pi' x \leq z  - \sqrt{\sigma}\label{ineq:EP}
\end{align}
with  $\pi_{(i)} = \sqrt{\sigma_{(i)}} - \sqrt{\sigma_{(i-1)}},  \ i=1,2,\ldots,n$,
is valid for the restriction with $y=x$, but not necessarily for $F$.
In this section, we lift inequality \eqref{ineq:EP} to obtain the linear valid inequality
\begin{align} \label{ineq:EPlift}
\pi' x \leq z + \alpha' (x-y) - \sqrt{\sigma}, 
\end{align}
for $F$ with coefficients $\alpha_{(i)} = {a_{(i)}}/{\sqrt{\sigma_{(i)}}}, \ i=1,2,\ldots,n$.

\begin{prop}
\label{prop:validity} 
Inequality \eqref{ineq:EPlift} with $\alpha$ and $\pi$ defined as above is valid for $F$.
\end{prop}

\begin{proof}
Consider the optimization problem over $F$:
\begin{align*}
\zeta = \max \ &  \pi' x - \alpha' (x-y) - z + \sqrt{\sigma} \\
  \text{s.t. } & \sigma + \sum_{i \in N} a_i y_i^2 \le z^2 \\
                              & \0 \leq y \leq x \\
                              & x \in \{0,1\}^N, \ y \in \R_+^N, \ z \in \R_+.
\end{align*}
Inequality \eqref{ineq:EPlift} is valid for $F$ iff $\zeta \le 0$.
By plugging in the values for $\pi, \ \alpha$ and eliminating $z$, the problem is equivalently written as
\begin{center}
\begin{minipage}{0.9\textwidth}
\leqnomode
\begin{align*}
\label{eq:validity} 
\zeta = \max \ &  \sum_{i \in [n]}\bigg (\sqrt{\sigma_{(i)}} - \sqrt{\sigma_{(i-1)}} -\frac{a_{(i)}}{\sqrt{\sigma_{(i)}}} \bigg ) x_{(i)} \\
\tag{$V$} \ & + \sum_{i \in [n]} \frac{a_{(i)}}{\sqrt{\sigma_{(i)}}} y_{(i)} - \sqrt{\sigma+ \sum_{i \in [n]}a_{(i)} y_{(i)}^2} + \sqrt{\sigma} \\
 \text{s.t. } &    \0 \leq y \leq x \\
 & x \in \{0,1\}^N, \ y \in \R_+^N.
\end{align*}
\reqnomode
\end{minipage}
\end{center}
Note that \eqref{eq:validity} is a special case of (\cP) with coefficients
\begin{align*}
c_{(i)} = - \bigg (\sqrt{\sigma_{(i)}} - \sqrt{\sigma_{(i-1)}} -\frac{a_{(i)}}{\sqrt{\sigma{(i)}}} \bigg ), \text { and } 
d_{(i)} =  -\frac{a_{(i)}}{\sqrt{\sigma{(i)}}}, \ i \in [n].
\end{align*}
Then
\[
 \frac{c_{(i)} + d_{(i)}}{a_{(i)}} = - \frac{\sqrt{\sigma_{(i)}} - \sqrt{\sigma_{(i-1)}} }{a_{(i)}}, \ i \in [n],
\]
\noindent
and, \rev{by concavity of the square root function,} we have
\begin{align*}
\frac{c_{(i)}+d_{(i)}}{a_{(i)}} \leq \frac{c_{(j)}+d_{(j)}}{a_{(j)}},  \ \text{ for } i \leq j.
\end{align*}
By \autoref{prop:consec1}, there exists an optimal solution $(x^*, y^*)$ to ($V$) such that  $x^* = \sum_{i=1}^{m} e_{(i)}$ for some $m \in [n]$. Then, $y^*$ is an optimal solution to the following convex problem:

\begin{center}
\begin{minipage}{0.99\textwidth}
\leqnomode
\begin{align*}
%\label{eq:validity_m}  
\begin{split}
\max \ &  \sum_{i \in [m]} (\sqrt{\sigma_{(i)}} - \sqrt{\sigma_{(i-1)}}) - 
		\sum_{i \in [m]} \frac{a_{(i)}}{\sqrt{\sigma_{(i)}}} (1 - y_{(i)}) - \sqrt{\sigma + \sum_{i \in [m]} a_{(i)} y_{(i)}^2} + \sqrt{\sigma} \\
 \text{s.t. } & \0 \leq y \leq \1. 
\end{split}
\end{align*}
\reqnomode
\end{minipage}
\end{center}

\noindent
\rev{This convex optimization problem over the continuous variables $y$ is a special case of (COPT), considered in the Appendix,} and its KKT conditions \rev{(following from  \eqref{FOC1}--\eqref{CS})} are satisfied by $({y}, {\lambda}, {\mu})$ such that
\begin{align*}
{y}_i& = 1, \  i \in [m], \\
{\lambda}_i &= 0, \  i \in [m], \\
{\mu}_i &= \frac{a_{(i)}}{\sqrt{\sigma_{(i)}}} - \frac{a_{(i)}}{\sqrt{\sigma_{(m)}}} \geq 0, \  i \in [m].
%\bar{\mu}_i &= \frac{a_i}{\sqrt{q(i)}} - \frac{a_i}{\sqrt{\sum_{i \in [m]} a_i y_i^2}} = \frac{a_i}{\sqrt{q(i)}} - \frac{a_i}{\sqrt{q(m)}} \geq 0, \  i \in [m].
\end{align*}
Therefore, there exists $(x^*, y^*) = (\sum_{i \in [m]} e_i, \sum_{i \in [m]} e_i)$ for some $m \in [n]$ with a binary $y^*$, implying $\zeta = 0$, i.e., the validity of \eqref{ineq:EPlift}.
\end{proof}

\begin{remark} Observe that
the proof of \autoref{prop:validity} implies that inequality \eqref{ineq:EPlift} is tight for the following $n+1$ 
affinely independent points of $F$:
\begin{align*}
(x,y,z) &= (\mathbf{0}, \mathbf{0}, \sqrt{\sigma}); \\
(x,y,z) &= \Big(\sum_{k \leq i}e_{(k)}, \sum_{k \leq i}e_{(k)}, \sqrt{\sigma_{(i)}} \Big), \ \ i \in [n]. \\
%(x,y,z) &= (e_1, 0.5 e_1, c_1 + 0.5 d_1 +  0.5\sqrt{q_1})
\end{align*}
\end{remark}

\ignore{, $c = [8, 5, 20, 11, 12]$, and $d = [-12, -6, -22, -12, -14]$.}
\ignore{which has the optimal solution
	\begin{align*}
	x^* = y^* =  \begin{bmatrix}
	1, 0, 1, 0, 1
	\end{bmatrix},
	z^* = -0.2540.
	\end{align*}
}

%The following example illustrates the generation of inequality \eqref{ineq:EPlift}.
\begin{example}
\label{ex:EPlift}
Consider an instance of $F$ with $a = [22, 18, 21, 19, 17], \ \sigma = 0$ and
the following fractional point is contained in its continuous relaxation: 
\begin{align*}
\bar{x} = \bar{y} = \begin{bmatrix}
1, 0.3817, 0.6543, 0.3616, 0.8083
\end{bmatrix},
\bar{z} = 6.8705.
\end{align*}
For the permutation $(1,3,5,2,4)$, $\pi$ is computed as
\begin{align*}
\pi_{(1)} &= \pi_1 =  \sqrt{a_1} - \sqrt{0} = \sqrt{22} = 4.6904, \\
\pi_{(2)} &= \pi_3 =  \sqrt{a_1+a_3} - \sqrt{a_1} = \sqrt{43} - \sqrt{22} = 1.8670, \\
 & \vdots \\
\pi_{(5)} &= \pi_4 = \sqrt{a_1+ \cdots + a_5} - \sqrt{a_1 + a_2 + a_3 + a_5 } = \sqrt{97} - \sqrt{78} = 1.0171. 
\end{align*}
The lifting coefficients $\alpha$ are computed accordingly, and we get inequality \eqref{ineq:EPlift} with
\begin{align*}
\pi &= \begin{bmatrix}
4.6904, 1.0858, 1.8670, 1.0171, 1.1885
\end{bmatrix}, \\
\alpha &= \begin{bmatrix}
4.6904, 2.0381, 3.2025, 1.9292, 2.1947
\end{bmatrix}.
\end{align*}
The fractional point $(\bar{x}, \bar{y}, \bar{z})$ is cut off by  \eqref{ineq:EPlift}
as $\pi ' \bar{x} - \alpha ' (\bar{x} - \bar{y}) - \bar{z} = 0.7844 > 0 $.
\end{example}

Although inequalities \eqref{ineq:EPlift} cut off points of the continuous relaxation with fractional $x$, unlike for the binary case
$K_\sigma$, adding all $n!$ inequalities \eqref{ineq:EPlift} 
is not sufficient to describe \conv(F) as \conv(F) is not a polyhedral set.
Therefore, in the next two subsections we present two nonlinear convex generalizations of inequalities \eqref{ineq:EPlift}.

\subsection{Lifted Nonlinear Polymatroid Inequalities I}
\label{sec:lift2}

The second class of lifted inequalities is obtained by applying the procedure described in Section~\ref{sec:lift1} 
for a subset of the variables. For $S \subseteq N$, introducing an auxiliary variable $t \in \R_+$, let us rewrite 
the conic constraint $\sum_{i\in N} a_i y_i^2  \le z^2$ as 

\begin{center}
\begin{minipage}{0.8\textwidth}
\leqnomode
\begin{align*}
%\label{eq:dcqip_S} \tag{$\cP_S$}
\begin{split}
 & {t^2 + \sum_{i\in N\setminus S} a_i y_i^2}  \le z^2, \\
                & \sigma + {\sum_{i \in S} a_i y_i^2} \le t^2. 
%                & \0 \leq y \leq x \\
%                & x \in \{0,1\}^N, y \in \R_+^N, \ t, z  \in \R_+.
\end{split}
\end{align*}
\reqnomode
\end{minipage}
\end{center}

Applying \autoref{prop:validity} to the relaxation defined by constraints
$$\sigma + {\sum_{i \in S} a_i y_i^2} \le t^2, \ \0 \leq y_S \leq x_S,  \ x_S \in \{0,1\}^S, y_S \in \R_+^S, t \in \R_+ $$
for each permutation $((1), (2), \ldots, (|S|))$ of $S$,
we generate a lifted polymatroid inequality \eqref{ineq:EPlift} of the form
\begin{align*}
\pi_S ' x_S \leq t + \alpha_S ' (x_S - y_S) - \sqrt{\sigma}
\end{align*}
where ${\pi_S}_{(i)} = \sqrt{\sigma_{S(i)}} - \sqrt{\sigma_{S(i-1)}}$, ${\alpha_S}_{(i)} = a_{(i)} / \sqrt{\sigma_{S(i)}}$, 
and the partial sums are defined as
$\sigma_{S(i)} = a_{(i)} + \sigma_{S(i-1)} \text{ for } i=1,2,\ldots,|S|$ with $\sigma_{S(0)} = \sigma$.
Eliminating the auxiliary nonnegative variable $t$, we obtain the following class of valid inequalities for $F$.

\begin{prop}
	For $S \subseteq N$, the inequality 
	\begin{align}
	\label{ineq:EP_subset}
	\big ((\pi_S ' x_S + \sqrt{\sigma} - \alpha_S ' (x_S - y_S))^+ \big)^2 + \sum_{i\in N \setminus S} a_i y_i^2 \le z^2
	\end{align}
	with $\pi_S$ and $\alpha_S$ defined above is valid for $F$.
\end{prop}

Note that inequality \eqref{ineq:EP_subset} is convex as it can be represented as conic quadratic by re-introducing the auxiliary variable $t \ge 0$ .
It is equivalent to \eqref{ineq:EPlift} for $S = N$ and to the original constraint for $S = \emptyset$. 
Otherwise, it is distinct from both. It is differentiable at $(x,y,z)$ with
$\pi_S ' x_S + \sqrt{\sigma} < \alpha_S ' (x_S - y_S)$ or $\pi_S ' x_S + \sqrt{\sigma} > \alpha_S ' (x_S - y_S)$.

\begin{remark}
The following $n+1$ affinely independent points of $F$
satisfy inequality \eqref{ineq:EP_subset} at equality:
\begin{align*}
(x,y,z) &= (\mathbf{0}, \mathbf{0}, \sqrt{\sigma}); \\
(x,y,z) &= \Big(\sum_{k \le i }e_{(k)}, \sum_{k \le i }e_{(k)}, \sqrt{\sigma_{(i)}} \Big), \  i = 1, 2, \ldots, |S|; \\
(x,y,z) &= (e_i, e_i, \sqrt{\sigma + a_i}), \ i \in N \setminus S.
%(x,y,z) &= (e_1, 0.5 e_1, c_1 + 0.5 d_1 +  0.5\sqrt{q_1})
\end{align*}
\end{remark}

The following example illustrates a point satisfying inequality \eqref{ineq:EPlift}, but 
cut off by inequality \eqref{ineq:EP_subset}.

\begin{example}
\label{ex:EPsubset}
Consider the instance in \autoref{ex:EPlift}, and the fractional point
\begin{align*}
\bar{x} = \bar{y} = (1, 0, 0, 0, 0.8), \ \bar{z} = 5.7341.
\end{align*}
This point satisfies inequality \eqref{ineq:EPlift} generated in \autoref{ex:EPlift}. 
Now letting $S = \{ 1, 2, 5 \}$ and using the permutation (1,5,2), $\pi_S$ is computed as
\begin{align*}
{\pi_S}_{(1)} &= \pi_1 = \sqrt{a_1} - \sqrt{0} = \sqrt{22}  = 4.6904, \\
{\pi_S}_{(2)} &= \pi_5 = \sqrt{a_1 + a_5} - \sqrt{a_1} = \sqrt{39} - \sqrt{22}  = 1.5546, \\
{\pi_S}_{(3)} &= \pi_2 = \sqrt{a_1 + a_5 + a_2 } - \sqrt{ a_1 + a_5} = \sqrt{57} - \sqrt{39}  = 1.3048. 
\end{align*}
Consequently, we obtain inequality \eqref{ineq:EP_subset} with coefficients
\begin{align*}
\pi &= \begin{bmatrix}
4.6904, 1.3048, 0, 0, 1.5546
\end{bmatrix}, \\
\alpha &= \begin{bmatrix}
4.6904, 2.3842, 0, 0, 2.7222
\end{bmatrix}.
\end{align*}
Observe that the fractional point $(\bar{x}, \bar{y}, \bar{z})$ is cut off by inequality \eqref{ineq:EP_subset} as
\[ \sqrt{ \big ((\pi_S' \bar{x} + \sqrt{\sigma}- \alpha_S'(\bar{x} - \bar{y}))^+ \big)^2 + \sum_{i \in N \setminus S} a_i \bar{y}_i^2 } - \bar{z} = 0.2 > 0. \]
\end{example}

\subsection{Lifted Nonlinear Polymatroid Inequalities II}
\label{sec:lift3}

The third class of inequalities are derived from a partial restriction of $F$ by setting a subset of the continuous variables to their upper bound. For $S \subseteq N$ and $T \subseteq N \setminus S$, consider the restriction of $F$ with  $y_i = x_i, \ i \in S$:
\begin{align*}
%\label{eq:dcqip_Sb} \tag{$\cP_{Sb}$}
 %& z \geq \sum_{i \in S} (c_i+d_i)x_i + \sum_{i \in N \setminus S} (c_i x_i + d_i y_i ) + t\\
    \ \ \            & t^2 + \sum_{i\in N \setminus (S \cup T)} a_i y_i^2 \le z^2 \\
        \ \ \            & \sigma + \sum_{i\in S} a_i x_i + \sum_{i\in T} a_i y_i^2 \le t^2 \\
\ignore{(F_S) \quad \quad   }             & y_i \leq x_i, \ i \in N \setminus S \\
                & x \in \{0,1\}^N, \ y \in \R_+^N, \ t \in \R_+.
\end{align*}

Applying the mixed-integer inequality \eqref{eq:polymatroidBoundedDominated} to the second constraint above, we obtain inequality
\begin{align*} %\label{eq:interim}
\pi_S ' x_S + \sqrt{\sigma + \sum_{i \in T} a_i y_i^2} \leq t, % \ \pi_S \in \Pi_{a(T)},
\end{align*}
where ${\pi_S}_{(i)} = \sqrt{\sigma_{S(i)}} - \sqrt{\sigma_{S(i-1)}}$ 
%${\alpha_S}_{(i)} = a_{(i)} / \sqrt{\sigma_{S(i)}+a(T)}$, 
and %the partial sums are defined as
$\sigma_{S(i)} = a_{(i)} + \sigma_{S(i-1)} \text{ for } i=1,2,\ldots,|S|$ with $\sigma_{S(0)} = \sigma + a(T)$.
This inequality is valid for the restriction above, but not necessarily for $F$.
Next, we lift it and eliminate the auxiliary nonnegative variable $t$, to obtain
the third class of valid inequalities
\begin{align}
\label{ineq:EP_mixed}
\bigg (\big (\pi_{S} ' x_S + \sqrt{\sigma + \sum_{i \in T} a_i y_i^2}  - \alpha_S (x_S - y_S) \big )^+ \bigg )^2  + \sum_{i\in N \setminus (S \cup T)} a_i y_i^2 \leq z^2,
\end{align}
for $F$ with $\alpha_{(i)} = {a_{(i)}}/{\sqrt{\sigma_{(i)}}}, \ i=1,2,\ldots,|S|$.

\begin{prop}
Inequality \eqref{ineq:EP_mixed} with $\alpha_S$ and $\pi_S$ defined as above is valid for $F$.
\end{prop}

\begin{proof} 
	It suffices to prove the validity of inequality
	\begin{align} \label{mixed-1}
	\pi_{S} ' x_S + \sqrt{\sigma + \sum_{i \in T} a_i y_i^2}  - \alpha_S (x_S - y_S) \le t.
	\end{align}
	Consider the optimization problem:
	\begin{align*}
	\zeta = \max \ &  \pi_S' x_S - \alpha_S' (x_S-y_S) + \sqrt{\sigma + \sum_{i \in T} a_i y_i^2} -t \\
	\text{s.t. } & \sigma +  \sum_{i \in N} a_i y_i^2 \le t^2 \\
	& \0 \leq y \leq x \\
	& x \in \{0,1\}^N, \ y \in \R_+^N, \ t \in \R_+.
	\end{align*}
	Inequality \eqref{mixed-1} is valid for $F$ iff $\zeta \le 0$.
	Observing that $x_i^* = y_i^* = 0$ for $i \in N \setminus (S \cup T)$ for an optimal solution $(x^*,y^*)$ and
    eliminating $t$, the problem is written as
	\begin{center}
		\begin{minipage}{0.9\textwidth}
			\leqnomode
			\begin{align*}
			\label{eq:validity} 
			\zeta = \max \ &  \pi_S' x_S - \alpha_S' (x_S-y_S) + \sqrt{\sigma + \sum_{i \in T} a_i y_i^2} - \sqrt{\sigma + \sum_{i \in S \cup T}a_{i} y_{i}^2}  \\
			 \text{s.t. } &    \0 \leq y \leq x \\
			& x \in \{0,1\}^N, \ y \in \R_+^N.
			\end{align*}
			\reqnomode
		\end{minipage}
	\end{center}
	Observe that by concavity of the square root function 
	 we have $y^*_i = 1$, $i \in T$. 
	The validity of 
	\[
	\pi_{S} ' x_S  - \alpha_S (x_S - y_S) \le t - \sqrt{\sigma}
	\]
	with $\bar \sigma = \sigma + a(T)$ and $t \ge \!\sqrt{\bar \sigma + \sum_{i \in S} a_i y_i^2}$ for this restriction implies that $\zeta \le 0$.
\end{proof}

%Note that each term on the left-hand-side is either conic quadratic or linear,
%which means that the resulting inequality is convex.
Note that when $S = \emptyset$ and $T = N$, \eqref{ineq:EP_mixed} is equivalent to the original constraint. 
When $S = N$, \eqref{ineq:EP_mixed} is equivalent to \eqref{ineq:EPlift}. 
When $S \subseteq N$ and $T = \emptyset$, \eqref{ineq:EP_mixed} is equivalent to \eqref{ineq:EP_subset}. 
Otherwise, it is the distinct from the three.

\begin{remark}
The following $n+1$ affinely independent points of $F$
satisfy inequality \eqref{ineq:EP_subset} at equality:
\begin{align*}
(x,y,z) &= (\mathbf{0}, \mathbf{0}, \sqrt{\sigma}); \\
(x,y,z) &= \Big(\sum_{k \le i}e_{(k)} + \sum_{j \in T} e_j, \sum_{k \le i}e_{(k)} + \sum_{j \in T} e_j, \sqrt{\sigma_{S(i)}} \Big), \ 
i = 1,2, \ldots, |S|;\\
(x,y,z) &= (e_i, e_i, \sqrt{\sigma + a_i}), \ i \in N \setminus S .
\end{align*}
\end{remark}

The following example illustrates inequality \eqref{ineq:EP_mixed} cutting off a fractional point
that is not cut by the previous inequalities.

\begin{example}
\label{ex:EPmixed}
Consider again the instance in \autoref{ex:EPlift}, and the fractional point
\begin{align*}
\bar{x} = \bar{y} = (0.8, 0.5, 1, 0, 1), \ \bar{z} = -0.1780.
\end{align*}
Note that this point satisfies inequalities \eqref{ineq:EPlift} and \eqref{ineq:EP_subset} generated in \autoref{ex:EPlift} and \autoref{ex:EPsubset}.
Letting $S = \{ 1, 2 \}$ and $T = \{3,5\}$, we have $a(T) = a_3 + a_5 = 38$. For the permutation (1,2), $\pi_S$ is computed as
\begin{align*}
{\pi_S}_{(1)} &= \pi_1 = \sqrt{a_1 + a(T)} - \sqrt{a(T)} = \sqrt{60} - \sqrt{38} = 1.5816, \\
{\pi_S}_{(2)} &= \pi_2 = \sqrt{a_1 + a_2 + a(T)} - \sqrt{a_1 + a(T)} =  \sqrt{78} - \sqrt{60} = 1.0858
\end{align*}
and we arrive at the corresponding inequality  \eqref{ineq:EP_mixed} with coefficients
\begin{align*}
\pi_S &= \begin{bmatrix}
1.5816, 1.0858, 0, 0, 0
\end{bmatrix}, \\
\alpha_S &= \begin{bmatrix}
2.8402, 2.0381, 0, 0, 0
\end{bmatrix}.
\end{align*}
Observe the point $(\bar{x}, \bar{y}, \bar{z})$ is cut off by \eqref{ineq:EP_mixed} as
\[ \sqrt{ \bigg ( \big (\pi_{S} ' \bar{x}_S + \sqrt{\sigma + \sum_{i \in T} a_i \bar{y}_i^2} - \alpha_S ' (\bar{x}_S - \bar{y}_S) \big )^+ \bigg )^2 + \sum_{i \in N \setminus (S \cup T)} a_i \bar{y}_i^2 }- \bar{z} = 0.4506 > 0. \]
\end{example}

\section{Computational Experiments}
\label{sec:computation}

In this section, we report the result of computational experiments performed to test the effectiveness of inequalities
\eqref{ineq:EPlift}, \eqref{ineq:EP_subset}, and \eqref{ineq:EP_mixed} in strengthening the continuous relaxation of mean-risk problems with on-off constraints. Three types of problems are used for testing: mean-risk problem with fixed-charges, mean-risk problem with a cardinality constraint, as well as the more general mean-risk problem with correlations and cardinality constraint.

All experiments are done using CPLEX 12.6.2 solver on a workstation with a 2.93GHz Intel
R CoreTM i7 CPU and 8 GB main memory and with a single thread. The time limit is set to two hours and CPLEX’ default settings are used with two exceptions: dynamic search is disabled to utilize the cut callbacks and the nodes are solved with the linear outer approximation for faster enumeration with node warm starts. The inequalities are added at nodes with depth less than ten.

\subsubsection*{Gradient cuts}
Recall that inequalities \eqref{ineq:EPlift} are linear; however, inequalities \eqref{ineq:EP_subset} and \eqref{ineq:EP_mixed} are (convex) non-linear. Since only linear cuts can be added using CPLEX callbacks, at a differentiable point $(\bar{x}, \bar{y})$, instead of a nonlinear cut $f(x,y) \le z$, we add the corresponding gradient cut
\begin{align*}
f(\bar{x}, \bar{y}) + [\nabla_x f(\bar{x}, \bar{y})]'(x - \bar{x}) + [\nabla_y f(\bar{x}, \bar{y})]' (y - \bar{y}) \le z.
\end{align*}
The gradient cut for inequality \eqref{ineq:EP_subset} at $(\bar{x}, \bar{y})$ with $\pi'_S \bar x_S + \sqrt{\sigma} > \alpha'_S(\bar x_S - \bar y_S)$ has the following form:
\begin{align}
\label{ineq:EP_subset_grad}
\kappa_1 + \frac{1}{f_1(\bar{x}, \bar{y})} \left [ \tau_1(\bar x, \bar y)  \big [ \pi_S'x_S - \alpha_S' (x_S - y_S) \big ]
+ \! \! \sum_{i \in N \setminus S} \! a_i \bar{y}_i y_i  \right ] \le z,
\end{align}
where
\begin{align*}
f_1(x,y) &=  \sqrt{\tau_1(x_S,y_S)^2 + \sum_{i\in N \setminus S} a_i y_i^2}, \\
\tau_1(x_S,y_S) &= \pi_S ' x_S + \sqrt{\sigma} - \alpha_S ' ({x_S} - {y_S}),
\end{align*}
and
\begin{align*}
\kappa_1 & = f_1(\bar{x}, \bar{y}) - \frac{\tau_1(\bar{x}_S, \bar{y}_S)\big[\tau_1(\bar{x}_S, \bar{y}_S) - \sqrt{\sigma}\big] + \sum_{i \in N \setminus S} a_i \bar{y_i}^2 }{f_1(\bar{x},\bar{y})} \cdot % \left[ \right].
\end{align*}

\noindent 
Observe that $\kappa_1$ is a constant that equals to zero when $\sigma = 0$.

The gradient cut for inequality \eqref{ineq:EP_mixed} at $(\bar{x}, \bar{y})$ 
with $\pi'_S \bar x_S + \sqrt{ \sigma \! + \! \sum_{i \in T} a_i \bar y_i^2} > \alpha'_S(\bar x_S - \bar y_S)$ 
has the form: 
\begin{align}
\label{ineq:EP_mixed_grad} 
\kappa_2 + \frac{1}{f_2(\bar{x}, \bar{y})}  \left [ \tau_2(\bar{x}, \bar{y}) \bigg [ \pi_S 'x_S - \alpha_S'(x_S - y_S) + \sum_{i \in T} \frac{a_i \bar{y}_i}{\nu(\bar{y})} y_i \bigg ] 
+ \! \sum_{i \in N \setminus (S \cup T)} \! a_i \bar{y}_i y_i \right ]    \le   z ,
\end{align}
where
\begin{align*}
f_2(x,y) &=  \sqrt{\tau_2(x_S,y_{S \cup T})^2  + \sum_{i\in N \setminus (S \cup T)} a_i y_i^2}, \\ 
\tau_2(x_S,y_{S \cup T}) &= \pi_{S} ' x_S + \nu(y_T) - \alpha_S' (x_S - y_S),\\
\nu(y_T) &=   \sqrt{\sigma + \sum_{i \in T} a_i {y}_i^2}, 
\end{align*}
and
\begin{align*}
\kappa_2 &= f_2(\bar{x}, \bar{y}) - 
\frac{\tau_2(\bar{x}_S, \bar{y}_{S \cup T}) 
	\big[ \pi_S' \bar{x}_S \! - \! \alpha_S (\bar{x}_S \! - \! \bar{y}_S) \! + \! \sum_{i \in T} \frac{a_i}{\nu(\bar y_T)} \bar{y}_i^2 \big] \! + \! \sum_{i \in N \setminus (S \cup T)} a_i \bar{y}_i^2 }{f_2(\bar{x}, \bar{y})} \cdot
\end{align*}

\noindent 
Observe that $\kappa_2$ is a constant that equals to zero when $\sigma = 0$.

\subsubsection*{Separation} 

The separation problem for inequalities \eqref{eq:extendedPolymatroidInequality} and $\conv(K_\sigma)$ is solved exactly and fast due to Edmond's greedy algorithm for optimization over polymatroids. We do not have such an exact separation algorithm for the lifted polymatroid inequalities and, therefore, use an inexact approach.

Given a point $(\bar{x}, \bar{y}, \bar{z})$, the separation for inequalities \eqref{ineq:EPlift} and $\conv(F)$ entails finding a permutation 
of $N$ for which the violation is maximized. If $\bar{x} = \bar{y}$, as it is the case for optimal 
solutions of the continuous relaxation of $(\cP)$ (see Appendix~\ref{subsec:solveR}), inequality \eqref{ineq:EPlift} coincides with the original polymatroid inequalitiy \eqref{eq:extendedPolymatroidInequality}. Therefore, we check the violation of inequality \eqref{ineq:EPlift} generated for a permutation $\left( (1), \ldots, (n) \right )$ satisfying
$\bar{x}_{(1)}\geq \bar{x}_{(2)} \geq \ldots \geq \bar{x}_{(n)}$.
If inequality \eqref{ineq:EPlift} is violated, then it is added to the formulation. Otherwise,
we attempt to find violated inequalities \eqref{ineq:EP_subset} and \eqref{ineq:EP_mixed} for the same permutation.
 
For inequality \eqref{ineq:EP_subset},
starting from $S = N$, we check for $i = (n), \ldots, (1)$ such that $\bar{x}_i - \bar{y}_i > 0$, 
whether moving $i$ from $S$ to $N \setminus S$ results in a violated inequality. If so, the corresponding gradient cut \eqref{ineq:EP_subset_grad} is added to the formulation. Similarly, for inequality \eqref{ineq:EP_mixed}, starting from thus constructed $S$, we check for  $i = (1), \ldots, (|S|)$ such that $\bar{x}_i - \bar{y}_i > 0$, whether moving $i$ from $S$ to $T$ results in a violated inequality.  If so, the corresponding gradient cut \eqref{ineq:EP_mixed_grad} is added to the formulation.
This heuristic is repeated for two additional permutations of $N$: one such that 
$a_{(1)}\bar{x}_{(1)}\geq a_{(2)}\bar{x}_{(2)} \geq \cdots \geq a_{(n)}\bar{x}_{(n)}$, and the other such that 
$a_{(1)}/\bar{x}_{(1)}\geq a_{(2)}/\bar{x}_{(2)} \geq \cdots \geq a_{(n)}/\bar{x}_{(n)}$. 
Throughout the branch-and-bound algorithm, the entire cut generation process is applied up to $5,000$ times for the first permutation, and up to $500$ times for the two additional permutations.

%The average of these numbers for each instance category are reported in Tables 
%Up to 1000 cuts at the root node and up to 100 cuts at each branch-and-bound tree nodes
%with depth less than 10 are allowed, and subproblems based on LP relaxation are solved throughout the algorithm.

%Each of the succeeding subsections report the result of computational experiments for each of the three types of problems we considered.

\subsection{Fixed-charge objective}
The first set of experiments are done on an optimization problem with fixed charges.
Each non-zero $y_i$ has a fixed-cost $c_i$, $i \in N$, which is modeled with cost vector $c$ 
on the binary indicator variables $x$:
\label{subsec:pen}
\begin{center}
\begin{minipage}{0.8\textwidth}
\leqnomode
\begin{align*}
\label{eq:pen} \tag{$\cP_f$}
\begin{split}
 %z^p = \ 
 \min & \ c'x + d'y + \Phi^{-1}(1-\epsilon)z \\
\text{s.t. } & \sum_{i \in N} a_i y_i^2 \le z^2 \\
                & \0 \leq y \le x  \\
                & x \in \{0,1\}^N, y \in \R_+^N, z \in \R_+.
\end{split}
\end{align*}
\end{minipage}
\end{center}
%where $c_i > 0, \ d_i < 0, \ a_i > 0 \ \forall i \in N$.

Five random instances are generated for each combination of confidence level $1-\epsilon \in \{0.9, 0.95, 0.975 \}$ 
and size $n \in \{100, 300, 500\}$.
Coefficients $a_i, \ i \in N$, are drawn from  integer uniform $[0.9n,1.2n]$, $c_i, \ i \in N,$ are drawn from  integer uniform $[5, 20]$. Finally,
for $i \in N$, $d_i$ is set to $-c_i - h_i$, where $h_i$ is drawn from integer uniform $[1,4]$. The data used for the experiments is publicly available for download at \texttt{http://ieor.berkeley.edu/$\sim$atamturk/data/} .

We compare the original and the strengthened formulations in Table~\ref{tb:pen_3cuts}. Each row of the table 
presents the averages for five instances. We report the
percentage integrality gap at the root node (rgap), solution time (time) in CPU seconds,
the percentage gap between the best upper bound and lower bound at termination (egap), and the number of nodes explored (nodes). 
The number of cuts generated for each type is also reported. 
If there are instances not solved to optimality within the time limit, we report their number  (\#) next to egap.

%\pagebreak

\ignore{
\begin{table}[b]
\footnotesize
\centering
\caption{Computations with $\cP_f$.}
\label{tb:pen_3cuts}
\setlength{\tabcolsep}{1pt} 
\begin{tabular}{c|c|rrrr|rrrrrrr}
\hline \hline
\multicolumn{2}{c|}{ } & \multicolumn{4}{c|}{ Default} & \multicolumn{7}{c}{With cuts} \\
\hline
$n$ & $1-\epsilon$ & rgap & time & egap (\#) & nodes & rgap & time & egap (\#) & nodes & cuts: \eqref{ineq:EPlift} &  \eqref{ineq:EP_subset} &  \eqref{ineq:EP_mixed} \\
\hline
\multirow{3}{*} {100}
 & 0.9 & 3.0 & 6,099 & 1.4 (4) & 237,470 & 0.0 & 0 & 0.0\phantom{ (5)}  & 0  & 14 & 0 & 0  \\
 & 0.95 & 10.9 & 7,200 & 8.6 (5) & 166,954 & 0.0 & 0 & 0.0\phantom{ (5)}  & 0 & 70 & 0 & 0  \\
 & 0.975 & 30.7 & 7,200 & 26.2 (5) & 226,365 & 0.0 & 0 & 0.0\phantom{ (5)} & 0 & 82 & 15 & 0  \\
\hline
\multirow{3}{*} {300}
 & 0.9 & 4.0 & 7,200 & 3.8 (5) & 125,005 & 0.0 & 1  & 0.0\phantom{ (5)}  & 0 & 49 & 0 & 0  \\
 & 0.95 & 12.8 & 7,200 & 12.7 (5) & 120,597 & 0.0 & 29 & 0.0\phantom{ (5)}  & 0 & 425 & 6 & 0  \\
 & 0.975 & 31.4 & 7,200 & 31.4 (5) & 134,433 & 0.0 & 23 & 0.0\phantom{ (5)}   & 0 & 437 & 36 & 1  \\
\hline
\multirow{3}{*} {500}
 & 0.9 & 3.9 & 7,200 & 3.8 (5) & 108,684 & 0.0 & 5  & 0.0\phantom{ (5)}   & 0 & 83 & 0 & 0  \\
 & 0.95 & 12.1 & 7,200 & 12.1 (5) & 112,765 & 0.0  & 253 & 0.0\phantom{ (5)}   & 0 & 693 & 12 & 0  \\
 & 0.975 & 31.3 & 7,200 & 31.3 (5) & 177,891 & 0.0  & 913  & 0.0\phantom{ (5)} & 0 & 1,119 & 158 & 2 \\
\hline 
\multicolumn{2}{c|}{\textbf{avg} }  &$ \BD{15.6}$& $\BD{7078}$ & $\BD{14.6}$\phantom{ (5)}&$\BD{156685}$ 
& $\BD{0.0}$ &$\BD{136}$ & $\BD{0.0}$\phantom{ (5)} &$\BD{0}$&$\BD{330}$&$\BD{25}$&$\BD{0}$ \\

\hline \hline
\end{tabular}
\end{table}
}

\begin{table}[b]
	\footnotesize
	\centering
	\caption{Computations with OPT$_f$.}
	\label{tb:pen_3cuts}
	\setlength{\tabcolsep}{1pt} 
	\begin{tabular}{c|c|rrrr|rrrrrrr}
		\hline \hline
		\multicolumn{2}{c|}{ } & \multicolumn{4}{c|}{ Default} & \multicolumn{7}{c}{With cuts} \\
		\hline
		$n$ & $1-\epsilon$ & rgap & time & egap (\#) & nodes & rgap & time & egap (\#) & nodes & cuts: \eqref{ineq:EPlift} &  \eqref{ineq:EP_subset} &  \eqref{ineq:EP_mixed} \\
		\hline
		\multirow{3}{*} {100}
		& 0.9 & 3.0 & 6,099 & 1.4 (4) & 237,470 & 0.0 & 0 & 0.0\phantom{(5)}  & 0  & 14 & 0 & 0  \\
		& 0.95 & 10.9 & 7,200 & 8.6 (5) & 166,954 & 0.0 & 0 & 0.0\phantom{(5)}  & 0 & 70 & 0 & 0  \\
		& 0.975 & 30.7 & 7,200 & 26.2 (5) & 226,365 & 0.0 & 0 & 0.0\phantom{(5)} & 0 & 82 & 15 & 0  \\
		\hline
		\multirow{3}{*} {300}
		& 0.9 & 4.0 & 7,200 & 3.8 (5) & 125,005 & 0.0 & 1  & 0.0\phantom{(5)}  & 0 & 49 & 0 & 0  \\
		& 0.95 & 12.8 & 7,200 & 12.7 (5) & 120,597 & 0.0 & 29 & 0.0\phantom{(5)}  & 0 & 425 & 6 & 0  \\
		& 0.975 & 31.4 & 7,200 & 31.4 (5) & 134,433 & 0.0 & 23 & 0.0\phantom{(5)}   & 0 & 437 & 36 & 1  \\
		\hline
		\multirow{3}{*} {500}
		& 0.9 & 3.9 & 7,200 & 3.8 (5) & 108,684 & 0.0 & 5  & 0.0\phantom{(5)}   & 0 & 83 & 0 & 0  \\
		& 0.95 & 12.1 & 7,200 & 12.1 (5) & 112,765 & 0.0  & 253 & 0.0\phantom{(5)}   & 0 & 693 & 12 & 0  \\
		& 0.975 & 31.3 & 7,200 & 31.3 (5) & 177,891 & 0.0  & 913  & 0.0\phantom{(5)} & 0 & 1,119 & 158 & 2 \\
		\hline 
		\multirow{3}{*} {1000}
		& 0.9 & 3.9 & 7,200 & 3.9(5) & 87,363 & 0.0 & 48 & 0.0\phantom{(0)} & 0 & 197 & 0 & 0 \\
		& 0.95 & 12.4 & 7,200 & 12.4(5) & 96,394 & 0.1 & 7,200 & 0.1(5) & 73 & 1,657 & 47 & 0 \\
		& 0.975 & 30.9 & 7,200 & 30.9(5) & 110,276 & 0.3 & 7,200 & 0.3(5) & 2 & 1,631 & 25 & 0 \\
		\hline
		\multicolumn{2}{c|}{\textbf{avg} }  &$ \BD{15.6}$& $\BD{7108}$ & $\BD{14.9}$\phantom{()}&$\BD{142016}$ 
		& $\BD{0.0}$ &$\BD{1306}$ & $\BD{0.0}$\phantom{()} &$\BD{6}$&$\BD{538}$&$\BD{25}$&$\BD{0}$ \\
		\hline \hline
	\end{tabular}
\end{table}

One observes in Table~\ref{tb:pen_3cuts} that the cuts have a profound effect in solving problem \eqref{eq:pen}. 
With the default setting, only one of \rev{60 instances is solved to optimality within two hours.  
The optimality gap reduces from $15.6\%$ to $14.9\%$ after exploring 142,016 nodes on average.  
On the other hand, when the cuts are added using the separation procedure outlined above,
50 of the 60 instances are solved at the root node without the need for enumeration. For the 10 instances that are not provably solved to optimality, the optimality gap is merely 0.2\% compared to 21.6\% with the default version for the same instances.}
%The average solution time is reduced by 98$\%$ from almost two hours to merely 136 seconds. 

\subsection{Cardinality constraint}
\label{subsec:card}
The second problem type with binary indicator variables has a cardinality constraint on the maximum non-zero $y_i, \ i \in N$. %of $x$ cannot exceed some number $\kappa n$ for $\kappa \in [0,1]$.
\begin{center}
\begin{minipage}{0.8\textwidth}
\leqnomode
\begin{align*}
\label{eq:card} \tag{$\cP_c$}
\begin{split} % z^c = \ 
\min & \  d' y + \ \Phi^{-1}(1-\epsilon)z \\
\text{s.t. } & \sum_{i \in N} a_i y_i^2 \le z^2\\
                & \sum_{i \in N} x_i \leq \kappa n \\
                & \0 \leq y \le x  \\
& x \in \{0,1\}^N, y \in \R_+^N, z \in \R_+.
\end{split}
\end{align*}
\end{minipage}
\end{center}
Instances are tested with two cardinality levels ($\kappa = 0.2, 0.4$).
Other parameters are generated as before.

The result of computations for \eqref{eq:card} is summarized in \autoref{tb:card_3cuts}.
Although the root gap for this type of problem is smaller compared to the fixed-charge objective problem, \rev{
only 27 out of 180 instances are solved to optimality using the default setting. 
When the cuts are utilized, the average root gap is reduced by 70$\%$ and 126 of the 180 instances are solved to optimality
within the time limit. 
%The largest end gap for the three unsolved instances is merely 0.05$\%$.
Accordingly, the average solution time as well as the number of nodes explored is reduced by orders of magnitude.
}

%\pagebreak

\ignore{
\begin{table}[t!]
\footnotesize
\centering
\caption{Computations with $\cP_c$.}
\label{tb:card_3cuts}
\setlength{\tabcolsep}{1pt} 
\begin{tabular}{c|c|c|rrrr|rrrrrrr}
\hline \hline
\multicolumn{3}{c}{ } & \multicolumn{4}{|c|}{ Default} & \multicolumn{7}{c}{With cuts} \\
\hline
$n$ & $\kappa$ & $1-\epsilon$ & rgap & time & egap (\#) & nodes & rgap  & time & egap (\#) & nodes & cuts: \eqref{ineq:EPlift} & \eqref{ineq:EP_subset} & \eqref{ineq:EP_mixed} \\
\hline
\multirow{6}{*} {100} & \multirow{3}{*}{0.4}
 & 0.9 & 0.2 & 8 & 0.0\phantom{ (5)} & 4,312 & 0.0  & 0  & 0.0\phantom{ (5)} & 0  & 5  & 0 & 0     \\
 & & 0.95 & 0.4 & 148 & 0.0\phantom{ (5)} & 23,884 & 0.0  & 0  & 0.0\phantom{ (5)}  & 0  & 16 & 0 & 0      \\
 & & 0.975 & 0.6 & 2,011 & 0.0\phantom{ (5)} & 110,692 & 0.0 & 0 & 0.0\phantom{ (5)} & 0 & 26 & 0 & 0      \\
\cline{2-14}
& \multirow{3}{*}{0.2}
	 & 0.9 & 1.1 & 718 & 0.0\phantom{ (5)} & 104,967 & 0.0 & 0 & 0.0\phantom{ (5)}  & 0 & 58 & 0 & 0     \\
 & & 0.95 & 2.2 & 4,838 & 0.5 (2) & 357,476 & 0.0 & 1  & 0.0\phantom{ (5)}  & 1  & 116  & 0 & 0     \\
 & & 0.975 & 3.7 & 7,200 & 1.8 (5) & 417,491 & 0.2 & 2 & 0.0\phantom{ (5)}  & 2 & 225 & 0 & 0     \\
\hline
\multirow{6}{*} {300} & \multirow{3}{*}{0.4}
	 & 0.9 & 0.4 & 7,200 & 0.3 (5) & 231,909 & 0.0 & 1 & 0.0\phantom{ (5)}  & 0  & 62  & 0 & 0     \\
 & & 0.95 & 0.7 & 7,200 & 0.6 (5) & 246,384 & 0.0  & 1  & 0.0\phantom{ (5)}  & 0  & 97  & 0 & 0     \\
 & & 0.975 & 1.0 & 7,200 & 1.0 (5) & 224,057 & 0.0  & 2  & 0.0\phantom{ (5)}  & 0  & 128  & 4 & 0    \\
\cline{2-14}
& \multirow{3}{*}{0.2}
	 & 0.9 & 1.9 & 7,200 & 1.7 (5) & 276,720 & 0.0 & 22 & 0.0\phantom{ (5)}  & 0  & 367  & 7 & 0    \\
 & & 0.95 & 3.2 & 7,200 & 3.1 (5) & 251,031 & 0.1 & 110 & 0.0\phantom{ (5)}  & 18 & 912 & 169 & 0    \\
 & & 0.975 & 4.8 & 7,200 & 4.6 (5) & 265,820 & 0.3 & 463 & 0.0\phantom{ (5)} & 116  & 1,565 & 680 & 2    \\
\hline
\multirow{6}{*} {500} & \multirow{3}{*}{0.4}
	 & 0.9 & 0.5 & 7,200 & 0.4 (5) & 191,359 & 0.0 & 23 & 0.0\phantom{ (5)}  & 0  &252  & 0 & 0    \\
 & & 0.95 & 0.8 & 7,200 & 0.8 (5) & 178,825 & 0.0  & 32  & 0.0\phantom{ (5)}  & 0  & 302  & 8 & 0     \\
 & & 0.975 & 1.1 & 7,200 & 1.1 (5) & 170,622 & 0.0 & 59 & 0.0\phantom{ (5)}  & 0  & 414  & 49 & 0     \\
\cline{2-14}
& \multirow{3}{*}{0.2}
	 & 0.9 & 2.0 & 7,200 & 2.0 (5) & 208,206 & 0.0 & 487 & 0.0\phantom{ (5)} & 10 & 1,140  & 51 & 27     \\
 & & 0.95 & 3.4 & 7,200 & 3.3 (5) & 215,453 & 0.1  & 2,372  & 0.0 (1) & 10,521  & 1,934 & 543 & 2   \\
 & & 0.975 & 4.9 & 7,200 & 4.9 (5) & 216,386 & 0.3 & 5,090 & 0.0 (2) & 5,669 & 3,283 & 1,725 & 8   \\
\hline
\multicolumn{3}{c|}{\textbf{avg} }  &$ \BD{1.8}$& $\BD{5629}$ & $\BD{1.4}$\phantom{ (5)}&$\BD{205311}$
& $\BD{0.1}$ &$\BD{481}$&$\BD{0.0}$\phantom{ (.)}&$\BD{908}$&$\BD{606}$&$\BD{180}$&$\BD{2}$  \\
\hline \hline
\end{tabular}
\end{table}
}

\begin{table}
	\footnotesize
	\centering
	\caption{Computations with OPT$_c$.}
	\label{tb:card_3cuts}
	\setlength{\tabcolsep}{1pt} 
	\begin{tabular}{c|c|c|rrrr|rrrrrrr}
		\hline \hline
		\multicolumn{3}{c}{ } & \multicolumn{4}{|c|}{ Default} & \multicolumn{7}{c}{With cuts} \\
		\hline
		$n$ & $\kappa$ & $1-\epsilon$ & rgap & time & egap (\#) & nodes & rgap  & time & egap (\#) & nodes & cuts: \eqref{ineq:EPlift} & \eqref{ineq:EP_subset} & \eqref{ineq:EP_mixed} \\
		\hline
		\multirow{9}{*} {100} & \multirow{3}{*}{0.4}
		& 0.9 & 0.2 & 8 & 0.0\phantom{(5)} & 4,312 & 0.0  & 0  & 0.0\phantom{(5)} & 0  & 5  & 0 & 0     \\
		& & 0.95 & 0.4 & 148 & 0.0\phantom{(5)} & 23,884 & 0.0  & 0  & 0.0\phantom{(5)}  & 0  & 16 & 0 & 0      \\
		& & 0.975 & 0.6 & 2,011 & 0.0\phantom{(5)} & 110,692 & 0.0 & 0 & 0.0\phantom{(5)} & 0 & 26 & 0 & 0      \\
		\cline{2-14}
		& \multirow{3}{*}{0.2}
		& 0.9 & 1.1 & 718 & 0.0\phantom{(5)} & 104,967 & 0.0 & 0 & 0.0\phantom{(5)}  & 0 & 58 & 0 & 0     \\
		& & 0.95 & 2.2 & 4,838 & 0.5(2) & 357,476 & 0.0 & 1  & 0.0\phantom{(5)}  & 1  & 116  & 0 & 0     \\
		& & 0.975 & 3.7 & 7,200 & 1.8(5) & 417,491 & 0.2 & 2 & 0.0\phantom{(5)}  & 2 & 225 & 0 & 0     \\
		\cline{2-14}
		& \multirow{3}{*}{0.1}
		& 0.9 &  4.7 & 2,993 & 0.1(1) & 247,352 & 1.8 & 2 & 0.0\phantom{(0)} & 12 & 326 & 0 & 0 \\
		& & 0.95 &   8.4 & 7,200 & 3.0(5) & 461,193 & 3.8 & 8 & 0.0\phantom{(0)} & 26 & 695 & 0 & 0 \\
		& & 0.975 &   12.6 & 7,200 & 6.2(5) & 512,712 & 6.1 & 26 & 0.0\phantom{(0)} & 56 & 1,614 & 0 & 0 \\
		\hline
		\multirow{9}{*} {300} & \multirow{3}{*}{0.4}
		& 0.9 & 0.4 & 7,200 & 0.3(5) & 231,909 & 0.0 & 1 & 0.0\phantom{(5)}  & 0  & 62  & 0 & 0     \\
		& & 0.95 & 0.7 & 7,200 & 0.6(5) & 246,384 & 0.0  & 1  & 0.0\phantom{(5)}  & 0  & 97  & 0 & 0     \\
		& & 0.975 & 1.0 & 7,200 & 1.0(5) & 224,057 & 0.0  & 2  & 0.0\phantom{(5)}  & 0  & 128  & 4 & 0    \\
		\cline{2-14}
		& \multirow{3}{*}{0.2}
		& 0.9 & 1.9 & 7,200 & 1.7(5) & 276,720 & 0.0 & 22 & 0.0\phantom{(5)}  & 0  & 367  & 7 & 0    \\
		& & 0.95 & 3.2 & 7,200 & 3.1(5) & 251,031 & 0.1 & 110 & 0.0\phantom{(5)}  & 18 & 912 & 169 & 0    \\
		& & 0.975 & 4.8 & 7,200 & 4.6(5) & 265,820 & 0.3 & 463 & 0.0\phantom{(5)} & 116  & 1,565 & 680 & 2    \\
		\cline{2-14}
		& \multirow{3}{*}{0.1}
		& 0.9 &  5.9 & 7,200 & 5.3(5) & 292,785 & 1.3 & 1,689 & 0.0(1) & 10,338 & 2,779 & 461 & 1 \\
		& & 0.95 & 10.1 & 7,200 & 9.3(5) & 274,213 & 3.2 & 6,193 & 0.5(4) & 11,958 & 4,893 & 3,428 & 61 \\
		& & 0.975 &  14.7 & 7,200 & 13.9(5) & 271,462 & 5.5 & 7,201 & 2.0(5) & 5,235 & 5,000 & 3,604 & 70 \\
		\hline
		\multirow{9}{*} {500} & \multirow{3}{*}{0.4}
		& 0.9 & 0.5 & 7,200 & 0.4(5) & 191,359 & 0.0 & 23 & 0.0\phantom{(5)}  & 0  &252  & 0 & 0    \\
		& & 0.95 & 0.8 & 7,200 & 0.8(5) & 178,825 & 0.0  & 32  & 0.0\phantom{(5)}  & 0  & 302  & 8 & 0     \\
		& & 0.975 & 1.1 & 7,200 & 1.1(5) & 170,622 & 0.0 & 59 & 0.0\phantom{(5)}  & 0  & 414  & 49 & 0     \\
		\cline{2-14}
		& \multirow{3}{*}{0.2}
		& 0.9 & 2.0 & 7,200 & 2.0(5) & 208,206 & 0.0 & 487 & 0.0\phantom{(5)} & 10 & 1,140  & 51 & 27     \\
		& & 0.95 & 3.4 & 7,200 & 3.3(5) & 215,453 & 0.1  & 2,372  & 0.0 (1) & 10,521  & 1,934 & 543 & 2   \\
		& & 0.975 & 4.9 & 7,200 & 4.9(5) & 216,386 & 0.3 & 5,090 & 0.0 (2) & 5,669 & 3,283 & 1,725 & 8   \\
		\cline{2-14}
		& \multirow{3}{*}{0.1}
		& 0.9 &  6.2 & 7,200 & 6.1(5) & 237,470 & 1.4 & 5,718 & 0.4(3) & 1,113 & 4,520 & 2,606 & 2 \\
		& & 0.95 &   10.5 & 7,200 & 10.4(5) & 232,484 & 3.5 & 7,200 & 2.3(5) & 335 & 4,412 & 2,341 & 28 \\
		& & 0.975 &   15.2 & 7,200 & 15.1(5) & 205,308 & 5.9 & 7,200 & 4.6(5) & 646 & 4,250 & 2,092 & 5 \\
		\hline
		\multirow{9}{*} {1000} & \multirow{3}{*}{0.4}
		& 0.9 &  0.5 & 7,200 & 0.5(5) & 102,180 & 0.0 & 311 & 0.0\phantom{(0)} & 0 & 509 & 43 & 0 \\
		& & 0.95 &  0.8 & 7,200 & 0.8(5) & 99,684 & 0.0 & 550 & 0.0\phantom{(0)} & 0 & 762 & 0 & 0 \\
		& & 0.975 &  1.1 & 7,200 & 1.1(5) & 116,666 & 0.0 & 1,265 & 0.0\phantom{(0)} & 0 & 1,094 & 0 & 0 \\
		\cline{2-14}
		& \multirow{3}{*}{0.2}
		& 0.9 &  2.1 & 7,200 & 2.1(5) & 142,802 & 0.0 & 6,792 & 0.0(3) & 179 & 2,817 & 373 & 0 \\
		& & 0.95 & 3.5 & 7,200 & 3.5(5) & 200,241 & 0.2 & 7,200 & 0.2(5) & 28 & 2,680 & 368 & 0 \\
		& & 0.975 &   5.3 & 7,200 & 5.3(5) & 187,518 & 0.8 & 7,200 & 0.8(5) & 0 & 2,053 & 0 & 0 \\
		\cline{2-14}
		& \multirow{3}{*}{0.1}
		& 0.9 & 6.8 & 7,200 & 6.8(5) & 178,937 & 2.0 & 7,200 & 2.0(5) & 0 & 1,938 & 0 & 0 \\
		& & 0.95 &   11.4 & 7,200 & 11.4(5) & 205,205 & 4.5 & 7,200 & 4.5(5) & 0 & 1,828 & 0 & 0 \\
		& & 0.975 & 16.6 & 7,200 & 16.6(5) & 214,779 & 7.4 & 7,200 & 7.4(5) & 0 & 1,915 & 0 & 0 \\
		\hline
		\multicolumn{3}{c|}{\textbf{avg} }  &$ \BD{4.6}$& $\BD{6301}$ & $\BD{4.0}$\phantom{(5)}&$\BD{211091}$
		& $\BD{1.4}$ &$\BD{2467}$&$\BD{0.7}$\phantom{(.)}&$\BD{1285}$&$\BD{1527}$&$\BD{515}$&$\BD{6}$  \\
		\hline \hline
	\end{tabular}
\end{table}

\ignore{
\subsection{Threshold constraint}

\begin{center}
	\begin{minipage}{0.8\textwidth}
		\leqnomode
		\begin{align*}
		\label{eq:threshold} \tag{$\cP_t$}
		\begin{split} % z^c = \ 
		\min & \  d' y + \ \Phi^{-1}(1-\epsilon)z \\
		\text{s.t. } & \sum_{i \in N} a_i y_i^2 \le z^2\\
%		& \sum_{i \in N} x_i \leq \kappa n \\
		& \ell \circ x \leq y \le u \circ x \\
		& x \in \{0,1\}^N, y \in \R_+^N
		\end{split}
		\end{align*}
	\end{minipage}
\end{center}
%where $ d_i < 0, \ a_i >0 \ \forall i \in N$.

$\ell_i$ = 0.01, 0.05
$u_i$ = 0.1
}

%\pagebreak

\subsection{Correlated case with cardinality constraint}
\label{subsec:corr}
Finally, although the cuts are developed for the diagonal uncorrelated case, 
we test their effectiveness on the more general correlated case with a cardinality constraint. 
Using the reformulation introduced in \autoref{sec:intro}, we state the problem as
\begin{center}
\begin{minipage}{0.8\textwidth}
\leqnomode
\begin{align*}
\label{eq:corr} \tag{$\cP_{corr}$}
\begin{split}
% z^{corr} = \ 
 \min & \ d'y + \Phi^{-1}(1-\epsilon) z \\
\text{s.t. } 
 & y' V y \le s^2 \\
 & s^2 + \sum_{i \in N} a_i y_i^2 \le z^2 \\               
                & \sum_{i \in N} x_i \leq \kappa n \\
                & \0 \leq y \leq x \\
                & x \in \{0,1\}^N, \ y \in \R_+^N, z \in \R_+.
\end{split}
\end{align*}
\end{minipage}
\end{center}
The covariance matrix $V \in \mathbb{R}^{n \times n}$ is computed using a factor model $V = \rho EFE'$, 
where $E \in \mathbb{R}^{n \times m}$ represents the exposures and $F \in \mathbb{R}^{m \times m}$ the factor covariance with
$m = n/10$.
We use a scaling parameter $\rho$ to test the impact of the magnitude of correlations on the difficulty of the problem. 
Since the cuts are developed for the diagonal case, we expect them to perform well for small $\rho$.
To ensure positive semidefiniteness, $F$ is computed as $F = GG'$ for $G \in \mathbb{R}^{m \times m}$. 
Each $G_{ij}$, $i,j \in [m]$ is drawn from uniform $[-1, 1]$, and 
$E_{ij}$, $i \in [n], \ j \in [m]$ is drawn from uniform $[0,0.1]$
with probability $0.2$ and set to $0$ with probability $0.8$.   
All other parameters are generated as before.

\begin{table}[h!]
\footnotesize
\centering
\caption{Computations with $\cP_{corr}$.}%  with $n=50$, $\kappa = 0.4$}
\label{tb:corr50_3cuts}
\setlength{\tabcolsep}{1pt} 
\begin{tabular}{c|c|c|rrrr|rrrrrrr}
\hline \hline
\multicolumn{3}{c|}{ } & \multicolumn{4}{c|}{ Default } & \multicolumn{7}{c}{With cuts} \\
\hline
n & $\rho$ &  $1-\epsilon$ & rgap & time & egap (\#) & nodes & rgap  & time & egap (\#) & nodes & cuts: \eqref{ineq:EPlift} &  \eqref{ineq:EP_subset} & \eqref{ineq:EP_mixed} \\
\hline
\multirow{12}{*}{100} & \multirow{3}{*}{0.1}
& 0.9 & 1.2 & 3,039 & 0.1(2) & 97,969 & 0.0 & 0 & 0.0\phantom{(5)} & 0 & 55 & 0 & 0 \\
&& 0.95 & 2.3 & 6,814 & 0.7(4) & 228,397 & 0.0 & 1 & 0.0\phantom{(5)} & 0 & 114 & 0 & 0 \\
&& 0.975 & 3.7 & 7,200 & 2.2(5) & 260,025 & 0.2 & 2 & 0.0\phantom{(5)} & 9 & 219 & 0 & 0 \\
\cline{2-14} 
&\multirow{3}{*} {1} 
 &  0.9 & 1.1 & 3,028 & 0.1(2) & 101,237 & 0.0 & 0 & 0.0\phantom{(5)} & 0 & 61 & 0 & 0 \\
 & & 0.95 &2.3 & 7,200 & 0.8(5) & 201,105 & 0.0 & 1 & 0.0\phantom{(5)} & 1 & 120 & 0 & 0 \\
 & & 0.975 &3.6 & 7,200 & 2.2(5) & 216,742 & 0.2 & 3 & 0.0\phantom{(5)} & 9 & 268 & 0 & 0 \\
\cline{2-14}
&\multirow{3}{*} {10} 
 & 0.9 & 1.1 & 3,017 & 0.1(2) & 111,262 & 0.0 & 1 & 0.0\phantom{(5)} & 10 & 144 & 0 & 0 \\
 & & 0.95 & 2.2 & 7,200 & 0.74(5) & 229,966 & 0.0 & 1 & 0.0\phantom{(5)} & 6 & 160 & 0 & 0 \\
 & & 0.975 & 3.5 & 7,200 & 2.16(5) & 243,539 & 0.2 & 3 & 0.0\phantom{(5)} & 21 & 402 & 0 & 0 \\
%\cline{2-14}
%&\multirow{3}{*} {1} 
% & 0.9 & 1.03 & 2214 & 0.01 & 96557 & 0.29 & 59 & 0.01 & 601 & 1249 & 0 & 0 \\
% & & 0.95 & 2.08 & 7200 & 0.82(5) & 222526 & 0.26 & 23 & 0 & 368 & 1185 & 0 & 0 \\
% & & 0.975 & 3.47 & 7200 & 2.16(5) & 230472 & 0.47 & 31 & 0 & 372 & 1458 & 46 & 0 \\
\hline
\multirow{12}{*}{300} & \multirow{3}{*}{0.1}
& 0.9 &	      	      	  1.9 & 7,200 & 1.7(5) & 163,249 & 0.0 & 72 & 0.0\phantom{(5)} & 33 & 695 & 125 & 0 \\
& & 0.95 &                3.2 & 7,200 & 3.2(5) & 154,394 & 0.1 & 271 & 0.0\phantom{(5)} & 65 & 1075 & 384 & 14 \\
& & 0.975 &               4.8 & 7,200 & 4.7(5) & 138,175 & 0.3 & 709 & 0.0\phantom{(5)} & 233 & 1921 & 991 & 10 \\
\cline{2-14}
& \multirow{3}{*}{1}
& 0.9 &                   1.9 & 7,200 & 1.7(5) & 129,379 & 0.0 & 2,034 & 0.0(1) & 14,370 & 3,328 & 500 & 0 \\
& & 0.95 &                3.2 & 7,200 & 3.2(5) & 123,180 & 0.1 & 1,804 & 0.0(1) & 8,926 & 2,099 & 810 & 35 \\
& & 0.975 &               4.8 & 7,200 & 4.7(5) & 132,002 & 0.3 & 2,579 & 0.0(1) & 6,121 & 2,315 & 1,233 & 15 \\
\cline{2-14}
& \multirow{3}{*}{10}
& 0.9 &                   1.8 & 7,200 & 1.6(5) & 161,802 & 0.3 & 7,201 & 0.2(5) & 31,917 & 3,893 & 485 & 1 \\
& & 0.95 &                3.2 & 7,200 & 3.1(5) & 143,816 & 0.3 & 7,201 & 0.2(5) & 34,684 & 3,422 & 1,550 & 13 \\
& & 0.975 &               4.8 & 7,200 & 4.6(5) & 132,304 & 0.5 & 7,201 & 0.2(5) & 37,645 & 3,639 & 2,592 & 23 \\
%\cline{2-14}
%& \multirow{3}{*}{1}
%& 0.9 &                   1.52 & 7200 & 1.43(5) & 164932 & 1.47 & 7200 & 1.41(5) & 79037 & 2947 & 0 & 0 \\
%& & 0.95 &                2.84 & 7200 & 2.77(5) & 155146 & 1.94 & 7200 & 1.88(5) & 17411 & 5000 & 827 & 0 \\
%& & 0.975 &               4.18 & 7200 & 4.13(5) & 152298 & 2.2 & 7200 & 1.84(5) & 18577 & 5000 & 3093 & 17 \\
\hline
\multicolumn{3}{c|}{ \textbf{avg} }  &$ \BD{2.8}$& $\BD{6483}$ & $\BD{2.1}$\phantom{(5)}&$\BD{164919}$
 & $\BD{0.1}$ &$\BD{1616}$&$\BD{0.0}$\phantom{(5)}&$\BD{7447}$&$\BD{1329}$&$\BD{482}$&$\BD{6}$ \\
\hline \hline
\end{tabular}
\end{table}

\autoref{tb:corr50_3cuts} presents the  results for confidence levels $1-\epsilon \in \{0.9, 0.95, 0.975\}$, problem sizes $n \in \{100, 300\}$, and 
scaling factors $\rho \in \{0.1, 1, 10\}$. 
As in the case of \eqref{eq:card},
the cuts result in significant improvements.
Out of 90 instances, the number of unsolved instances is reduced from 80 to 18, and the average root gap is reduced by 96$\%$.
Especially, for instances with $\rho \in \{0.1, 1\}$, almost all instances are solved to optimality well within the time limit
and the number of nodes is reduced by an order of magnitude. Even for $\rho = 10$, the end gap is reduced from 3.1\% to only 0.2\%.
As expected, the computational results indicate that inequalities \eqref{ineq:EPlift}, \eqref{ineq:EP_subset}, and \eqref{ineq:EP_mixed} are more effective when the covariance matrix is more diagonal-dominant, i.e., for smaller values of $\rho$. The lifted polymatroid
cuts are, nevertheless, valuable for the general correlated case as well.

\exclude{
\autoref{tb:corr} summarizes the result
for \eqref{eq:corr} instances with $\kappa = 0.4$ and $\rho = 0.01$ for problem sizes $n = 100, 300$.
This choice of parameters results in 'easier' instances since $\kappa = 0.4$
imposes a weaker constraint than $\kappa = 0.2$, and scaling down the entries of matrix $V$ reduces the effect of
having more complex objective function. Nevertheless, the average root gap is much greater than those of both \eqref{eq:pen} and \eqref{eq:card}
and none of the instances were solved to optimality for $n = 300$ with or without the cuts. }

\section{Conclusion}
\label{sec:conclusion}

In this paper we study a mixed 0-1 optimization with conic quadratic objective arising when modeling utilities with risk averseness. 
Exploiting the submodularity of the underlying set function for the binary restrictions, 
we derive three classes of strong convex valid inequalities by lifting of the polymatroid inequalities.
Computational experiments demonstrate the effectiveness of the lifted inequalities in a cutting plane framework.
The results indicate that the inequalities are very effective in strengthening the convex relaxations, and thereby,
reducing the solution times problems with fixed charges and cardinality constraints substantially. Although the inequalities are derived for the diagonal case, they
are also effective in improving the convex relaxations for the general correlated case.

\ignore{However, the inequalities that we generated were not the strongest possible as we did not solve the separation problem exactly.
Further analysis of the separation problem may provide a heuristic that results in stronger cuts
than the empirical approach we took in this study. Also, although the lifted inequalities are computationally beneficial, they do not suffice in describing the convex hull of the original problem.
One of the goals of our future research is to identify the inequalities necessary in characterizing the convex hull which are likely to be nonlinear.
Another direction of future study would be examining our cuts for more complicated real-life problems
that has $F$ as a substructure.
}

\section*{Acknowledgement} This research is supported, in part, by grant FA9550-10-1-0168 from the Office
of the Assistant Secretary of Defense for Research and Engineering.

\bibliographystyle{plain}
\bibliography{QIPref}

%\pagebreak

\section*{Appendix}

%\subsection{Solving the continuous relaxation}
\label{subsec:solveR}

In this appendix we characterize the solutions to the continuous relaxation of the optimization problem (\cP)
and present an algorithm to solve it.  Let $(\tilde{x}, \tilde{y})$ denote an optimal solution to the relaxation.
As $c > 0$ it is clear that $\tilde{x}_i = \tilde{y}_i, i \in N$. Then letting $\tilde c = c + d$ ($< 0$ by assumption),
we can reduce the continuous relaxation of (\cP) 
to the following convex optimization problem:

\begin{align*}
%		\label{eq:rcqip} 
		\min \ \ \ & \tilde c' y + \sqrt{\sigma + \sum_{i \in N} a_i y_i^2}  \nonumber \\
(\ropt) \ \ \quad		\text{s.t. }  -&y_i  \leq 0, \ i \in N \qquad  (\lambda_i) \\
		& y_i  \leq 1, \ i \in N.  \qquad (\mu_i)
		%                -x_i & \leq 0 \ \forall i \in N & (\gamma_i) \label{con3}\\
		%                x_i & \leq 1 \ \forall i \in N & (\eta_i) \label{con4}
\end{align*}

Assume that the variables are indexed so that 
\[\frac{\tilde c_1}{a_1} \leq  \frac{\tilde c_2}{a_2} \leq \cdots \leq \frac{\tilde c_n}{a_n} \cdot \]

\begin{prop}  \label{prop:relax-opt}
If $\tilde y$ is an optimal solution to (\ropt), then $\tilde y_i \ge \tilde y_k$, $ 1\le i < k \le n$.
\end{prop}

\begin{proof}
Let $\lambda, \mu$ be the dual multipliers associated with the lower bound and upper bound constraints for $y$, respectively.
Since  (\ropt) is a convex optimization problem with linear constraints, the KKT conditions 
\begin{subequations}
%\textit{Dual feasibility}
\begin{align}
\tilde c_i + \frac{a_i}{\sqrt{\sigma + \sum_{j \in N} a_j y_j^2}} y_i - \lambda_i + \mu_i &= 0, \ i \in N \label{FOC1}
%-\mu_i - \gamma_i + \eta_i &= 0 \ \forall i \in N \label{FOC2}
\end{align}

%\textbf{Primal \& Dual Feasibility}
\begin{align}
\label{FEAS}
\begin{split}
\mathbf{0} \leq y \leq \mathbf{1} \\
\lambda, \mu \geq \mathbf{0}
\end{split}
\end{align}

%\textbf{Complementary Slackness}
\begin{align}
\label{CS}
\begin{split}
y_i \lambda_i &= 0, \ i \in N \\
(1-y_i) \mu_i &= 0, \ i \in N \\
%x_i \gamma_i &= 0 \\
%(1-x_i) \eta_i &= 0
\end{split}
\end{align}
\end{subequations}
are necessary and sufficient for optimality.

By complementary slackness (c.s.), observe that $\lambda_i$ and $\mu_i$ cannot be both positive. 
Therefore, there are three possible combinations of values for $\lambda_i$ and $\mu_i$, for $i \in N$:     
\begin{enumerate}
\item[1)] $\lambda_i > 0, \ \mu_i = 0$: 
c.s. implies $y_i = 0$, which by \eqref{FOC1} implies $\lambda_i = \tilde c_i < 0$, violating $\lambda_i \ge 0$.
Hence, this case is unattainable.
\item[2)] $\lambda_i = 0, \ \mu_i > 0$: c.s. implies $y_i = 1$. Then \eqref{FOC1} is written as
\begin{align*}
\mu_i = - \frac{a_i}{\sqrt{\sigma + \sum_{j \in N} a_j y_j^2}}   - \tilde c_i,
\end{align*}
which is dual feasible only if $-\frac{\tilde c_i}{a_i} \geq \frac{1}{\sqrt{\sigma + \sum_{j \in N} a_j y_j^2}} $.
\item[3)] $\lambda_i = \mu_i = 0$: 
In this case, \eqref{FOC1} reduces to
\begin{align*}
 - \tilde c_i = \frac{a_i}{\sqrt{\sigma + \sum_{j \in N} a_j y_j^2}} y_i ,
\end{align*}
which is feasible only if $-\frac{\tilde c_i}{a_i} \leq \frac{1}{\sqrt{\sigma +\sum_{j \in N} a_j y_j^2}} $.
\end{enumerate}

Hence, the assumption $\tilde c < 0$ implies $\tilde y > 0$ and, further, either one of the following holds for $i \in N$:
\begin{align*}
& (a) \ \tilde{y}_i = 1 \text{ and } \frac{ \tilde c_i}{a_i} \leq \frac{-1}{\sqrt{\tilde{\sigma}}}, \\
& (b) \ 0 < \tilde{y}_i < 1 \text{ and } \frac{\tilde c_i}{a_i} = \frac{-\tilde{y}_i}{\sqrt{\tilde{\sigma}}},
\end{align*}
where $\tilde{\sigma} := \sigma + \sum_{i \in N} a_i {\tilde{y}_i}^2.$ Then, if $\tilde{y}_k = 1$, since $\frac{\tilde c_i}{a_i} \le \frac{\tilde c_k}{a_k} \le \frac{-1}{\sqrt{\tilde{\sigma}}} $ for $i < k$, 
it follows that $\tilde{y}_i = 1$. 
If $\tilde{y}_k$ is fractional,  for $i < k$ 
\begin{align*}
\tilde{y}_k = -\frac{ \tilde c_k}{a_k} \sqrt{\tilde{\sigma}}
               \leq -\frac{ \tilde c_i}{a_i} \sqrt{\tilde{\sigma}}
\end{align*}
and, therefore, $\tilde{y}_k \le \tilde{y}_i = \min \Big\{1, -\frac{\tilde c_i}{a_i} \sqrt{\tilde{\sigma}}\Big\}$.
\end{proof}

There are two simple special cases where we can generate a closed form solution for KKT points.
\begin{remark}
If $- \tilde c_i > \frac{a_i}{\sqrt{\sigma + \sum_{i \in N} a_i}}, \ i \in N$, then $\tilde{y}_i = 1, i \in N$.
\end{remark}
\begin{remark}
If $\sigma + \sum_{i \in N} \frac{\tilde c_i^2}{a_i} = 1$ and $-\frac{\tilde c_i}{a_i} \leq 1, \ i \in N$,
then $\tilde{y}_i = -\frac{\tilde c_i}{a_i}, \ i \in N$.
\end{remark}

In the remainder, we give an algorithm that constructs a KKT point for (\ropt).
Defining the sets
$ N_1 := \{i \in N :  \tilde{y}_i = 1 \} $ and
$ N_f := \{i \in N : 0 < \tilde{y}_i < 1 \}$,
we can express $\tilde{\sigma}$ as
\begin{align*}
\tilde{\sigma} &= \sigma + \sum_{i \in N_1} a_i {\tilde{y}_i}^2 + \sum_{i \in N_f} a_i {\tilde{y}_i}^2 \\
       &= \sigma + \sum_{i \in N_1} a_i + \sum_{i \in N_f} a_i \left( -\frac{\tilde c_i}{a_i} \sqrt{\tilde{\sigma}}\right)^2   \\
       &= \sigma + \sum_{i \in N_1} a_i + \tilde{\sigma} \sum_{i \in N_f} \frac{ \tilde c_i^2}{a_i}. 
\end{align*}
Therefore, given $N_1$ and $N_f$, one can compute
\begin{align*}
\tilde{\sigma} (N_1, N_f) &= \frac{\sigma + \sum_{i \in N_1} a_i}{\left( 1 - \sum_{i  \in N_f} \frac{\tilde c_i^2}{a_i} \right)} \\
\tilde{y}_i &= -\frac{\tilde c_i}{a_i} \sqrt{\tilde{\sigma}(N_f, N_1)}, \  i \in N_f, \\
\tilde{y}_i &= 1, \ i \in N_1.
\end{align*}

\autoref{alg:KKT} describes how to construct $N_1$ and $N_f$.
Initially, $N_f = \emptyset$ and $N_1 = N$, i.e., $y_i = 1, \text{ for all } i \in N$.
At each iteration of \autoref{alg:KKT}, checks whether $\tilde{y}_p$ is fractional or one.
Either $p$ is moved from $N_1$ to $N_f$
and the incumbent over-estimation $\tilde{\sigma}(N_1, N_f)$ on $\tilde{\sigma}$ is updated accordingly, or it is determined that
$\tilde y_p = 1$ and the algorithm terminates as $\tilde y_i = 1, \text{ for all } i < p$ due to \autoref{prop:relax-opt}.
Observe that if the indices satisfy non-decreasing order of $\tilde c_i/a_i$, \autoref{alg:KKT} runs in $O(n)$ time.

\begin{algorithm}
\DontPrintSemicolon
\SetAlgoNoEnd\SetAlgoNoLine%
 {\texttt{0. \hskip -3mm Initialize}} \\ \texttt{Set} $N_f = \emptyset$, $N_1 = N$, $\tilde{\sigma} = \sigma + \sum_{i \in N} a_i$, $p =n $. \\
% Index the variables to satisfy $\frac{\tilde c_1}{a_1} \leq \frac{\tilde c_2}{a_2} \leq \cdots \leq \frac{\tilde c_n}{a_n} $. \;
 \hrule
 \bigskip
 \texttt{1. \hskip -3mm Update} \\
\quad \uIf{$-\frac{\tilde c_p}{a_p} \geq \frac{1}{\sqrt{\tilde{\sigma}}}$ or $p=0$}{ \texttt{go to step 2.}}
 \quad \Else{ $N_f \leftarrow N_f \cup \{p\}$ \;
           $N_1 \leftarrow N_1 \setminus \{p\}$ \;	
%            $\tilde{\sigma} \leftarrow \tilde{\sigma} - q_p + \frac{(c_p+d_p)^2}{q_p} \tilde{\sigma}$  \;
            % $\tilde{\sigma} \leftarrow \frac{\tilde{\sigma} - q_p}{\left( 1 - \frac{\tilde c_p^2}{q_p}\right)}$  \;
             $\tilde{\sigma} \leftarrow \tilde{\sigma} (N_1, N_f)$ \\ %\frac{\tilde{\sigma} - q_p}{\left( 1 - \frac{\tilde c_p^2}{q_p}\right)}$  \;
           $p \leftarrow p-1$  \\ \texttt{Repeat step 1.} }
\hrule
\bigskip
 \texttt{2. \hskip -3mm Terminate} \\
% $\tilde{\sigma} \leftarrow \frac{\sum_{i \in N_1} a_i}{\left( 1 - \sum_{i  \in N_f} \frac{\tilde c_i^2}{a_i} \right)} $ \;
%\uIf{$-\frac{\tilde c_p}{q_p} < \frac{1}{\sqrt{\tilde{\sigma}}}$ and $p>0$}{Go to step 1.}
%\Else{ 
	\quad \texttt{Return} \;
$ \quad \tilde{y}_i = -\frac{\tilde c_i}{a_i} \sqrt{\tilde{\sigma}}, \  i \in N_f $ \;
$\quad \tilde{y}_i = 1, \  i \in N_1$
 \caption{KKT point construction for (\ropt).}
 \label{alg:KKT}
\end{algorithm}

\end{document}